\documentclass[a4paper,onecolumn,amsmath,amssymb,amsfonts,accepted=2026-03-06]{quantumarticle}
\pdfoutput=1
\usepackage{subfig,dsfont,mathrsfs,amsthm,stmaryrd,mathtools}
\usepackage{braket}
\usepackage[T1]{fontenc}
\usepackage{hyperref}
\usepackage[capitalise]{cleveref}
\usepackage[utf8]{inputenc}
\usepackage{cancel}
\usepackage{algorithm}
\usepackage{algpseudocode}

\usepackage[colorinlistoftodos,prependcaption]{todonotes}
\usepackage[
    backend=biber,
    sorting=nyt,
    isbn=false,
    doi=true,
    url=false,
    giveninits=true,
    date=year]{biblatex}
\newcommand{\xL}{\mathcal{L}}
\newcommand{\Id}{\operatorname{\mathbf{Id}}}
\newcommand{\xH}{\mathcal{H}}

\newcommand{\xF}{\mathcal{F}} 
\newcommand{\xP}{\mathbf{Q}}
\newcommand{\xtr}[1]{\operatorname{Tr}\left(#1 \right)}

\newcommand{\xK}{\mathcal{K}}
\newcommand{\xN}{\mathbb{N}}

\newcommand{\cD}{\mathcal D}
\newcommand{\LL}{{\mathbf{\Gamma}}}

\newcommand{\Hamil}{{\bf H}}
\newcommand{\Proj}[1]{\mathbf{P}_{#1}}
\newcommand{\cA}{\mathcal A}
\newcommand{\qop}{{\bf q}}
\newcommand{\Pop}{{\bf p}}
\newcommand{\destroy}{{\bf a}}
\newcommand{\create}{{\bf a^\dag}}
\newcommand{\destroyb}{{\bf b}}
\newcommand{\createb}{{\bf b^\dag}}
\newcommand{\rN}{{\boldsymbol{\rho}_{(N)}}}
\newcommand{\Nop}{\mathbf{N}}
\newcommand{\Uop}{\mathbf{U}}

\newtheorem{lemma}{Lemma}
\newtheorem{remark}{Remark}
\newtheorem{definition}{Definition}
\newtheorem{proposition}{Proposition}
\newtheorem*{proposition*}{Proposition}
\newtheorem{corollary}{Corollary}
\newtheorem{example}{Example}
\newtheorem*{example*}{Example}

\title{A posteriori error estimates for the Lindblad master equation}
\author{Paul-Louis Etienney}
\email{paul-louis.etienney@proton.me}
\author{Rémi Robin}
\email{remi.robin@minesparis.psl.eu}
\author{Pierre Rouchon}
\email{pierre.rouchon@minesparis.psl.eu}
\affiliation{Laboratoire de Physique de l'\'Ecole Normale Supérieure, Mines Paris, Inria, CNRS, ENS-PSL, Sorbonne Universit\'{e}, PSL Research University, Paris, France}
\date{March 09, 2026}

\allowdisplaybreaks

\begin{document}
\maketitle

\begin{abstract}
We are interested in the simulation of open quantum systems governed by the Lindblad master equation in an infinite-dimensional Hilbert space. To simulate the solution of this equation, the standard approach involves two sequential approximations: first, we truncate the Hilbert space to derive a differential equation in a finite-dimensional subspace. Then, we use discrete time-step to obtain a numerical solution to the finite-dimensional evolution.

In this paper, we establish bounds for these two approximations that can be explicitly computed to guarantee the accuracy of the numerical results. Through numerical examples, we demonstrate the efficiency of our method, empirically highlighting the tightness of the upper bound. While adaptive time-stepping is already a common practice in the time discretization of the Lindblad equation, we extend this approach by showing how to dynamically adjust the truncation of the Hilbert space. This enables fully adaptive simulations of the density matrix. For large-scale simulations, this approach can significantly reduce computational time and relieves users of the challenge of selecting an appropriate truncation.
Furthermore, as a special case, our method naturally applies to Hamiltonian (unitary) dynamics.

\end{abstract}
\tableofcontents

\section{Introduction}
\label{subsec_lindblad}
\subsection{Presentation of the problem and main contributions}
In this article, we focus on simulating open quantum systems obeying the Lindblad master equation~\cite{lindbladGeneratorsQuantumDynamical1976, goriniCompletelyPositiveDynamical1976,gkls_history}. This equation models the evolution of an open quantum system weakly coupled to a Markovian environment \cite{breuerTheoryOpen2006}. We are particularly interested in the case where the underlying Hilbert space $\xH$ is infinite-dimensional, with a particular focus on bosonic modes, also known as quantum resonators, or cavities. 
More precisely, the Lindblad equation reads as follows:
\begin{align}
    \notag
    \xL(\boldsymbol{\rho})& =-i[\Hamil,\boldsymbol{\rho}] + \sum_i \cD_{\LL^i}(\boldsymbol{\rho}),  \\
    \frac{d}{dt} \boldsymbol{\rho}(t)&= \xL(\boldsymbol{\rho}(t)), \quad \boldsymbol{\rho}(0)=\boldsymbol{\rho}_0,
    \label{eq-lindblad_full}
\end{align}
where $\boldsymbol{\rho}$ is a density operator, that is a self-adjoint positive semidefinite operator of trace one, $\Hamil$ is the Hamiltonian of the system, which is a self-adjoint operator, the bracket denotes the commutator, and the dissipators $\cD_{\LL^i}(\boldsymbol{\rho})$ act on $\boldsymbol{\rho}$ as
\begin{equation}
	\cD_{\LL}(\boldsymbol{\rho})
	= \LL \boldsymbol{\rho} {\LL}^\dag - \frac12 \left( {\LL}^\dag\LL \, \boldsymbol{\rho} + \boldsymbol{\rho}\, {\LL}^\dag \LL\right).
    \label{eq-dissipator}
\end{equation}

In the case of infinite dimensional Hilbert space and unbounded operators $\LL^i$ and/or $\Hamil$, special care is needed to define the solution of the equation, we refer to~\cite{daviesGeneratorsDynamicalSemigroups1979,chebotarevSufficientConditionsConservativity1998} for the general case or \cite{gondolfEnergyPreservingEvolutions2023} for the specific case of bosonic modes.

We are interested in computing an approximation of the solution $\boldsymbol{\rho}(T)$ of \cref{eq-lindblad_full} at a final time $T>0$.
To this aim, two intermediate steps are classically performed.

\paragraph{Step 1.} We begin by selecting a finite-dimensional subspace $\xH_N \subset \xH$. Our goal is to define a finite-dimensional approximation of the solutions of \cref{eq-lindblad_full} within this subspace. Let $\Proj{N}$ denotes the orthogonal projector onto $\xH_N$. We consider the corresponding truncated operators:
\begin{align}
  \Hamil_N=\Proj{N} \Hamil \Proj{N},\quad {\LL^i}_N=\Proj{N} \LL^i \Proj{N}, \quad \forall i\in \xN.
\end{align}

For a single bosonic mode, the Hilbert space is given by $\xH = \{ \sum_{n=0}^\infty u_n \ket{n} \mid \sum_n |u_n|^2< \infty\} \simeq  l^2(\xN)$, where the canonical orthonormal basis is provided by the Fock states $(\ket{n})_{n \geq 0}$, and the corresponding dual basis is $(\bra{n})_{n \geq 0}$. In this setting, a common method is to truncate the Fock basis, namely consider
\begin{align}
    \xH_N=\operatorname{Span} \{\ket{n}\mid 0 \leq n  \leq N\}, \quad \Proj{N}=\sum_{i=0}^{N}\ket{i}\bra{i}.
\end{align}
Next, we define the following Lindbladian for both density operators on $\xH$ and on $\xH_N$.
\begin{align}
    \label{truncated_lindblad}
    \xL_N(\boldsymbol{\rho})& =-i[\Hamil_N,\boldsymbol{\rho}] + \sum_i \cD_{{\LL^i}_N}(\boldsymbol{\rho}).
\end{align}
Thus, we can define the approximated solution $\rN(t)$ by
\begin{align}
    \label{eq-lindblad_trunc-1}
    \frac{d}{dt} \rN(t)&= \xL_N(\rN(t)), \quad \rN(0)=\Proj{N}\boldsymbol{\rho}_0\Proj{N}.
\end{align}
It is important to keep in mind that $\Proj{N}\boldsymbol{\rho}(t)\Proj{N}$ is not equal in general to $\rN(t)$; but we expect $\rN(t)$ to be an approximation of $\boldsymbol{\rho}(t)$. For the sake of simplicity, we assume that $\boldsymbol{\rho}_0$ has support in $\xH_N$, that is $\boldsymbol{\rho}_0=\Proj{N}\boldsymbol{\rho}_0\Proj{N}$ (see \cref{rmk:initial state} for the general case).

\paragraph{Step 2:} Approximate the solution of the linear differential equation \cref{eq-lindblad_trunc-1} using an Ordinary Differential Equation (ODE) solver. In the following, we will denote by $\xF_{\delta t}$ one step of the considered scheme. Let us give an example with the first order explicit Euler scheme, which reads as follows
\begin{align}
    e^{\delta t \xL_N} (\boldsymbol{\rho}) \approx \xF_{\delta t}(\boldsymbol{\rho})= \boldsymbol{\rho} + \delta t \xL_N(\boldsymbol{\rho}).
\end{align}

Taking $N_{step}$ and $ \delta t$ such that $T=N_{step} \delta t$, we compute iteratively $\tilde{ \boldsymbol{\rho}}_{n+1,\delta t}= \xF_{\delta t}(\tilde{ \boldsymbol{\rho}}_{n,\delta t})$ with $\tilde{ \boldsymbol{\rho}}_{0,\delta t}=\boldsymbol{\rho}_0$, that is $\tilde{ \boldsymbol{\rho}}_{N_{step},\delta t}= \xF_{\delta t}^{N_{step}}(\boldsymbol{\rho}_0)$. At this point, we expect $\tilde{ \boldsymbol{\rho}}_{N_{step},\delta t}$ to be close to $\boldsymbol{\rho}(T)$.

\paragraph{}

These two steps are completely standard and can be easily performed using a library for quantum simulation like \textsc{Qutip} \cite{JOHANSSON20131234}, \textsc{Qiskit dynamics} \cite{qiskit_dynamics_2023}, \textsc{Dynamiqs} \cite{guilmin2024dynamiqs}, \textsc{QuantumOptics.jl} \cite{kramer2018quantumoptics}, \textit{etc.} The time solver scheme is usually more complex and powerful than a first order explicit Euler scheme. For instance, it can be a Runge-Kutta scheme or a structure preserving one, like \cite{cao2024structurepreservingnumericalschemeslindblad,appelo2024krauskinghighordercompletely,robin2025unconditionallystabletimediscretization}. The space discretization method we employed in step 1 is commonly known as a Galerkin approximation. For a detailed study of the convergence and limitation of this method to the Schrödinger equation, we refer to \cite{fischer2025quantumparticlewrongbox}. In this article, we do not consider non-linear space discretization methods, such as reduced rank or tensor network methods, nor do we address stochastic unraveling schemes and cumulant expansions.

\paragraph{}

This paper focuses on providing computable bounds on the distance between the approximations $\rN(T)$ or $\tilde{ \boldsymbol{\rho}}_{N_{step},\delta t}$ and the true solution $\boldsymbol{\rho}(T)$. The error $\|\rN(T)-\boldsymbol{\rho}(T)\|_1$ is referred to as the (space-)truncation error, while $\|\rN(T)-\tilde{ \boldsymbol{\rho}}_{N_{step},\delta t}\|_1$ is the time discretization error. Previous works, such as \cite{PRL_estimator}, have addressed (while not being the main contribution of the paper) the space truncation error for finite-dimensional Hamiltonian systems coupled to chains of bosonic modes, using time-adaptive density matrix renormalization group (t-DMRG) methods. Our approach differs in several important ways. First, our proposed method applies to systems that are themselves infinite-dimensional, such as bosonic modes. Second, it is suitable for Lindblad dynamics, not just closed systems. See \cref{sec:PRL_comp} for details.

Other works focus on \textit{a priori} error analysis, typically for specific Hamiltonian systems (e.g.\ \cite{Tong2022provablyaccurate,Peng_2025}); an exception is the recent study of Lindblad equations in \cite{robin2025convergenceanalysisgalerkinapproximations}. However, these studies do not provide explicit usable computable estimates, but rather asymptotic convergence rates. Even when constants can be found, they are typically not explicit and the resulting bounds are not sharp. Conceptually, \textit{a priori} estimates use bounds depending on the exact solution $\boldsymbol{\rho}(t)$ or its regularity (for example the norm of $ \boldsymbol{\rho}(t)$ in a well-chosen Sobolev space), whereas \textit{a posteriori} estimates only depend on the computed solution $\rN(t)$ or $\tilde{ \boldsymbol{\rho}}_{N_{step},\delta t}$. While the  first allows proving convergence and rates of convergence, the second is much sharper in practice to control the error of a given simulation.

Many examples considered in this work are motivated by the bosonic code community, a subfield of quantum error correction devoted to the robust encoding of quantum information in bosonic modes. In this setting, one encounters a broad class of Hamiltonians and dissipative generators, frequently expressed as polynomials in the creation and annihilation operators. In the experimental platform of circuit quantum electrodynamics, sum of unitary operators of the form $e^{i\eta \qop}+e^{-i\eta \qop}$ naturally arise, where $\qop$ denotes the position operator. Such terms originate from the Josephson junction nonlinearity, whose interaction energy is proportional to the cosine of the flux. The simulation of dissipative cat-qubits \cite{mirrahimiDynamicallyProtectedCatqubits2014}, as well as several proposals for implementations of Gottesman--Kitaev--Preskill (GKP) qubits \cite{GKP_original}, is often subtle. Indeed, these qubits deliberately exploit the large dimensionality of the underlying infinite-dimensional Hilbert space in order to encode quantum information with enhanced protection.

The techniques and estimates developed here are, however, not restricted to bosonic systems. They extend directly to other quantum many-body settings—for instance, to spin lattices—as briefly discussed at the end of \cref{subsec_spaceestimate}.

The organization of this article is as follows.

In \cref{subsec_notation}, we introduce notations. In \cref{sec:sec_1,sec_ex1,sec_ex2,sec_faster}, we assume that we have access to the continuous-in-time solution $(\rN(t))_{0 \leq t \leq T}$ of \cref{eq-lindblad_trunc-1}, that is we neglect the time discretization errors. This is motivated by the fact that it is common to have efficient adaptive high-order time integration schemes, so the primary source of difficult-to-control error is precisely the truncation error.

In \cref{sec:sec_1}, we provide an upper bound on the truncation error $\|\rN(T)-\boldsymbol{\rho}(T)\|_1$. This upper bound can be computed using the trajectory $(\rN(t))_{0\leq t \leq T}$ in many cases of interest. We mention the case of spin lattices.

In \cref{sec_ex1}, we apply this estimate to the case of bosonic modes with polynomials of creation and annihilation operators. In this case we first show that it is always possible to compute the bound given in \cref{sec:sec_1}. Then we study in more details several examples from the bosonic code community. We provide simulations illustrating the tightness of the estimate.

In \cref{sec_ex2}, we investigate more complex examples involving either unitary operators or dissipators of the form $e^{i\eta\qop} \left( \Id-\epsilon\Pop\right)$. The main difficulty is that for the truncated Fock basis $\xH_N$, $e^{i\eta\qop}\xH_N$ is not supported on a finite set of Fock states.

Then, in \cref{sec_faster}, we leverage these estimates to introduce a \textit{space-adaptive} method. While \textit{time-adaptive} stepping is common in practice, dynamically adapting the size of the Hilbert space is not. Hence, we show that our estimates offer numerous opportunities to enhance the efficiency and robustness of numerical schemes. This approach also relieves the user from the challenging task of manually determining an efficient truncation balancing accuracy and computational cost. These ideas are implemented in the new open source library \textsc{dynamiqs\_adaptive} developed by the first author.

In \cref{sec:time_estimator}, we do not neglect the time-discretization error anymore, and provide both \textit{a posteriori} estimates for time-discretization (that is $\|\rN(T)-\tilde{ \boldsymbol{\rho}}_{N_{step},\delta t}\|_1$) and for the cascade of the two approximations (i.e. $\|\boldsymbol{\rho}(T)- \tilde{ \boldsymbol{\rho}}_{N_{step},\delta t}\|_1$).

In \cref{sec:pathological}, we provide examples of dynamics where the approach of comparing $\rN(t)$ and $\boldsymbol{\rho}_{(N+K)}(t)$ fails, and in \cref{sec:PRL_comp}, we compare our method to the one presented in \cite{PRL_estimator}.

\subsection{Notations}
\label{subsec_notation}
In this paper, we use the following notations
\begin{itemize}
    \item $\xH$ denotes a separable (often infinite dimensional) Hilbert space.
    \item $\xH_N$ is a finite dimensional vector space embedded in $\xH$. $\Proj{N}$ is the orthogonal projector on $\xH_N$ and $\Proj{N}^\perp$ is the orthogonal projector on $\xH_N^\perp$. In particular, $\Proj{N}+\Proj{N}^\perp=\Id$. As an  example, for $m$ bosonic modes, we have $\xH= l^2(\xN)^{\otimes m}$, $N$ could be the multi-index $(N_1,\ldots,N_{m})$ and $\xH_N$ would be defined by
    \begin{align}
        \xH_N=\operatorname*{Span}\{ \ket{i_1}\otimes \ldots \otimes \ket{i_m}\mid 0\leq i_1 \leq N_1,\, \ldots,\, 0\leq i_m \leq N_m \}.
    \end{align}
    It is always assumed that $\xH_N$ is in the domain of the Lindbladian $\xL$. Besides, in \cref{sec:time_estimator}, we will also assume that  $\xH_N$ is in the domain of $\xL\circ \xL$ (and as many iterations required for higher-order schemes).
    \item $\mathcal{O}(\xH)$ (resp. $\mathcal{O}(\xH_N)$) denotes the set of operators on $\xH$ (resp. $\xH_N$). Bold letters are used to represent these operators.
    \item $\Id$ and $\Id_N$ denote the identity operator in resp. $\xH$ and $\xH_N$.
    \item For any operator $\mathbf{A}\in \mathcal{O}(\xH)$, we denote $\mathbf{A}_N=\Proj{N} \mathbf{A} \Proj{N}$ its truncation to $\xH_N$.
    \item $\xL_N$ denotes the Lindbladian on $\xH_N$ obtained by truncating the Hamiltonian and the dissipators, see \cref{truncated_lindblad}.
    \item $(\rN(t))_{0 \leq t \leq T}$ is the solution of the Lindblad equation with Lindbladian $\xL_N$ on $\xH_N$, see \cref{eq-lindblad_trunc-1}. Note that we often assume that $\rN(0)=\boldsymbol{\rho}_0\in \mathcal{O}(\xH_N)$, refer to \cref{rmk:initial state} for the general case.
    \item For a single bosonic mode, we recall that the annihilation operator $\destroy$ and its adjoint the creation operator $\create$ are defined as follows:
    \begin{align}
        \destroy= \sum_{n=1}^\infty\sqrt{n}\ket{n-1}\bra{n}, \quad \create = \sum_{n=0}^\infty\sqrt{n+1}\ket{n+1}\bra{n}.
    \end{align}
    The position and momentum operators obey the relations
    \begin{align}
    \qop &= \frac{{\destroy + \create}}{\sqrt{2}}, \quad \Pop = \frac{{\destroy - \create}}{i\sqrt{2}},
    \end{align}
    and $\Nop=\create \destroy$ denotes the photon number operator. For a system with two modes, that is $\xH=l^2(\xN)\otimes l^2(\xN)$, the annihilation operator on the first mode ($\destroy \otimes \Id$) and on the second mode ($\Id \otimes \ \destroy$) are denoted by $\destroy$ and $\destroyb$ with a surcharge of notation. 
    \item We recall that the trace norm (also known as nuclear norm) is defined by 
    \begin{align}
        \label{eq_nuclear_norm}
        \|\boldsymbol{\rho}\|_1= \xtr{|\boldsymbol{\rho}|}, \quad |\boldsymbol{\rho}| =\sqrt{\boldsymbol{\rho}^\dag \boldsymbol{\rho}}.
    \end{align}
    The set of trace-class operators $\xK^1(\xH)= \{ \boldsymbol{\rho} \in \xH \mid \|\boldsymbol{\rho}\|_1 <\infty\}$ equipped with the trace-norm is a Banach space. We denote $\xK^1_s(\xH)$ the linear set of self-adjoint operators in $\xK^1(\xH)$ and $\xK^1_+(\xH)\subset \xK^1_s(\xH)$ the convex cone of positive semidefinite operators. As $\dim(\xH_N)<\infty$, the set $\xK^1(\xH_N)$ coincides with $\mathcal{O}(\xH_N)$.
    
\end{itemize}

\section{\textit{A posteriori} truncation error estimates for Lindblad equation}
\label{sec:sec_1}
\label{subsec_spaceestimate}
In this section, we assume that we have access to the continuous-in-time solution $(\rN(t))_{0 \leq t\leq T}$ of \cref{eq-lindblad_trunc-1}. Note that for all $0 \leq t \leq T$, $\rN(t)\in \xK^1_+(\xH_N)$.

The key to obtain our estimates resides in the fact that the flow of the Lindblad equation contracts the trace norm. Let us first recall a classical result.
\begin{proposition}{\cite[Lemma 1]{KOSSAKOWSKI1972247}}
    \label{prop:CPTP}
    Let $\mathcal{M}$ be a Completely Positive Trace Preserving (CPTP) map, and $\boldsymbol{\sigma}\in \xK_s^1$. Then,
    \begin{align}
        \|\mathcal{M}(\boldsymbol{\sigma})\|_1 \leq \|\boldsymbol{\sigma}\|_1.
    \end{align} 
\end{proposition}
\begin{proof}
    $\boldsymbol{\sigma}$ is self-adjoint, thus we can decompose it into a positive part and negative part, so that $\boldsymbol{\sigma}=\boldsymbol{\sigma}_+-\boldsymbol{\sigma}_-$, with $\boldsymbol{\sigma}_+\geq 0$, $\boldsymbol{\sigma}_- \geq 0$ and $\boldsymbol{\sigma}_- \boldsymbol{\sigma}_+ = 0$. Thus,
    \begin{align}
        \|\boldsymbol{\sigma}\|_1 &= \|\boldsymbol{\sigma}_+ - \boldsymbol{\sigma}_-\|_1 = \xtr{\boldsymbol{\sigma}_+} + \xtr{\boldsymbol{\sigma}_-}.
    \end{align}
    Then, we compute
    \begin{align}
        \|\mathcal{M}(\boldsymbol{\sigma})\|_1 &= \|\mathcal{M}(\boldsymbol{\sigma}_+) - \mathcal{M}(\boldsymbol{\sigma}_-)\|_1\notag\\
        &\leq \|\mathcal{M}(\boldsymbol{\sigma}_+)\|_1+\|\mathcal{M}(\boldsymbol{\sigma}_-)\|_1 \quad\text{(triangle inequality)}\notag\\
        &= \operatorname{Tr}(\mathcal{M}(\boldsymbol{\sigma}_+)) + \operatorname{Tr}(\mathcal{M}(\boldsymbol{\sigma}_-)) \quad\text{($\mathcal{M}$ is completely positive)}\notag\\
        &= \operatorname{Tr}(\boldsymbol{\sigma}_+) + \operatorname{Tr}(\boldsymbol{\sigma}_-) \quad\text{($\mathcal{M}$ is trace preserving)}\notag\\
        & = \|\boldsymbol{\sigma}\|_1,
    \end{align}
    which concludes the proof.
\end{proof}

The flow of the Lindbladian, that is $\boldsymbol{\rho}_0 \mapsto e^{t\xL}\boldsymbol{\rho}_0$ for any $t\geq 0$, is CPTP. Hence, it contracts the trace norm.

\begin{lemma}
    \label{prop-aposteriori}
    We have the following estimate
    \begin{align}
        \label{eq_estimate_space}
        \|\boldsymbol{\rho}(t)-\rN(t)\| \leq \|\boldsymbol{\rho}(0)-\rN(0)\|_1 + \int_0^t \left\|\left(\xL-\xL_N\right) \rN(s)\right\|_1 ds.
    \end{align}
\end{lemma}
\begin{proof}
    The evolution of $\mathbf{r}(t)=\boldsymbol{\rho}(t)-\rN(t)$ can be written as
    \begin{align}
        \notag\frac{d}{dt}\mathbf{r}(t)&=\mathcal{L}(\boldsymbol{\rho}(t))-\mathcal{L}(\rN(t))+\mathcal{L}(\rN(t))-\mathcal{L}_N(\rN(t))\\
        \label{eq:errors}
        &=\mathcal{L}(\mathbf{r}(t))+(\mathcal{L}-\mathcal{L}_N)(\rN(t)).
    \end{align}
    Then, using Duhamel's principle, we get
    \begin{align}
        \label{eq-rt}
        \mathbf{r}(t)= e^{t\xL}\mathbf{r}(0) + \int_0^t  e^{(t-s)\xL} \left( \left(\xL-\xL_N\right) \rN(s)\right)ds.
    \end{align}
    Taking the trace-norm and applying the triangle inequality gives
    \begin{align}
        \|\mathbf{r}(t)\|_1 \leq  \|\mathbf{r}(0)\|_1 + \int_0^t  \left\|\left(\xL-\xL_N\right) \rN(s)\right\|_1 ds.
    \end{align}
\end{proof}
As a consequence, we can define the estimator $\xi$ as
\begin{align}
    \label{eq_estimator_xi}
    \dot \xi(t)= \|\left(\xL-\xL_N\right) \rN(t)\|_1, \quad \xi(0)= \| \boldsymbol{\rho}_0- \rN(0)\|_1,
\end{align}
and we get $\xi(t)\geq \|\boldsymbol{\rho}(t)-\rN(t)\|_1$ for all $t\geq 0$.
\begin{remark}
    \label{rmk:initial state}
   If $\boldsymbol{\rho}_0$ does not belong to $\xK^1_+(\xH_N)$, but only to $\xK^1_+(\xH)$, we use the unnormalized initial condition $\Proj{N}\boldsymbol{\rho}_0\Proj{N}\in \xK^1_s(\xH_N)$ as initial condition for $\rN(0)$. The first term on the right-hand-side of \cref{eq_estimate_space} accounts for this error.
\end{remark}

\begin{remark}
    \label{time:dependent}
   \cref{prop-aposteriori} holds for time-dependent Lindbladians as well. The proof is the same, except that we need to replace $e^{t\xL}$ (resp. $e^{(t-s)\xL}$) by the propagator of the time-dependent Lindbladian from time $0$ to time $t$ (resp. from time $s$ to time $t$). Indeed, the propagator of a time-dependent Lindbladian is also CPTP, and thus contracts the trace norm; we refer the reader to \cite[Theorem 3.5 and remark 3 after Theorem 2.11]{gondolfEnergyPreservingEvolutions2023} or \cite{article_chebo}.
\end{remark}

\cref{prop-aposteriori} provides an \textit{a posteriori} error estimate because the computation of $\|(\xL-\xL_N)(\rN(t))\|_1$ only requires the knowledge of the trajectory $(\rN(t))_{0 \leq t \leq T}$ (and not of $(\boldsymbol{\rho}(t))_{0 \leq t \leq T}$, which is usually intractable). To compute the estimator, two main components are required:

\paragraph{}First, we need the ability to compute $\|\left(\xL-\xL_N\right) \rN(t)\|_1$. For a general Lindbladian, this computation may not be feasible since $\xL(\rN(t))$ might be inaccessible\footnote{For instance, \(\xL(\rN(t))\) might have infinite rank, and is thus impossible to represent within a finite-dimensional linear space.}. However, in the following sections, we present a broad class of Lindbladians where it is possible to compute or bound cleverly \(\|\left(\xL - \xL_N\right) \rN(t)\|_1\). In \cref{sec_ex1}, we consider Lindbladians involving only polynomials in creation and annihilation operators. In this case $\xL(\rN)$ can always be explicitly computed as it has support in the slightly larger Hilbert space $\xH_{N+d}$ for some integer $d$. In \cref{sec_ex2}, we consider more complex examples where $\xL(\rN)$ cannot be explicitly computed but one can still obtain a good upper bound on $\|(\xL-\xL_N)(\rN(t))\|_1$.

\paragraph{}The second requirement is the ability to compute the integral in \cref{eq_estimate_space}. This can be approximately achieved by numerically solving \cref{eq_estimator_xi}. For a bound that does not involve time interpolation, we refer the reader to \cref{sec:time_estimator}.

\paragraph{Application to spin lattices}
One family of examples satisfying the first requirement is spin lattices in the thermodynamic limit, where the truncation retains only a finite number of spins. Consider a $d$-dimensional lattice indexed by $\mathbb{Z}^d$ (or a subset thereof), with each site $j$ having a finite-dimensional local Hilbert space $\mathbb{C}^{d_j}$ (e.g., $d_j = 2$ for spin-1/2 systems). Given a finite cubic domain $\Lambda_N \subset \mathbb{Z}^d$ of linear size $N$, the truncated Hilbert space is the finite tensor product $\xH_N = \bigotimes_{j \in \Lambda_N} \mathbb{C}^{d_j}$. The infinite system is understood in the thermodynamic limit as $N\to\infty$, or more rigorously in the operator-algebraic framework via the inductive limit of local algebras.

For physically relevant spin lattices, the Lindbladian $\xL$ is typically a sum of local terms: $\xL = \sum_{j} \xL_j$, where each $\xL_j$ acts non-trivially only on a finite number of neighboring spins (e.g., nearest-neighbor or finite-range interactions). The truncated Lindbladian $\xL_N$ retains only those terms $\xL_j$ that are supported in $\Lambda_N$. Crucially, the difference $(\xL - \xL_N)(\rN(t))$ involves only boundary terms—interactions crossing the boundary $\partial \Lambda_N$ of the truncation region. More precisely, if the interactions have range $r$ (i.e., $\xL_j$ couples at most $r$-th nearest neighbors), then $\xL(\rN(t))$ can be represented as an operator on a slightly larger finite region $\Lambda_{N+r}$. The error $\|(\xL-\xL_N)(\rN(t))\|_1$ can thus be numerically computed, as it only involves operators on the finite-dimensional space $\xH_{N+r}$.

\section{Application 1: estimates for polynomial operators on bosonic modes}
\label{sec_ex1}
\subsection{General case}
\label{subsec:ap1general_case}
In this section, we consider the case where $\xH=l^2{(\xN)}^{\otimes m}$, and the Hamiltonian and the dissipators are polynomials in creation and annihilation operators. We show that for these systems, the computation of $\xL(\rN(t))$ is simple. First, we treat the case of a single mode before discussing the generalization to several modes.

\begin{definition}
    An (unbounded) operator $\LL$ on $l^2(\xN)$ is a polynomial in the creation and annihilation operators of degree $d$ if there exists a (non-commutative) polynomial of degree $d\in \xN$ of the form ${Q}[X, Y]=\sum_{i+j\leq d}\nu_{i,j}X^iY^j$ such that $\LL={Q}[\destroy, \create]=\sum_{i+j\leq d}\nu_{i,j}\destroy^i \create^j$.
\end{definition}
If the Hamiltonian and the dissipators are polynomials in $\destroy$ and $\create$, the corollaries of the next proposition show that it is always possible to compute explicitly $\xL(\rN)$.
\begin{proposition}
    \label{proposition:polyink}
    Let ${Q}[X, Y]$ be a polynomial of degree $d\geq 1$, then
    \begin{align}
       {Q}[\destroy, \create] \xH_N \subset \xH_{N+d}.
    \end{align}
\end{proposition}
\begin{proof}
    Using that $[\destroy,\create]=\Id$, we can reduce ${Q}[\destroy, \create]$ to the following form
    \begin{align}
        {Q}[\destroy, \create]&= \sum_{i+2j \leq d} \lambda_{i,j} \destroy^i \Nop^j + \mu_{i,j}  \create^i \Nop^j,
    \end{align}
    where $\Nop=\create \destroy$ is the photon number operator. If we apply the operator ${Q}[\destroy, \create]$ on an element of $\xH_N$, the first part of the sum remains in $\xH_N$. As $\xH_N$ is invariant under the action of $\Nop$, $\create^i \Nop^j \xH_N \subset \xH_{N+i}$. Hence, ${Q}[\destroy, \create] \xH_N \subset \xH_{N+d}$.
\end{proof}
\begin{example}
    For any $\boldsymbol{\sigma}_N \in \mathcal{O}(\xH_N)$, we can decompose $\boldsymbol{\sigma}_N=\sum_{0\leq i,j\leq N} \sigma_{i,j} \ket{i}\bra{j}$, and we get indeed that $\create \boldsymbol{\sigma}$ belongs to $\mathcal{O}(\xH_{N+1})$.
\end{example}
\begin{corollary}
    \label{co:Hamiltonian}
    Let $\Hamil$ be a polynomial operator of degree $d$ in $\destroy$ and $\create$. Then, for every operator $\boldsymbol{\sigma}_N\in \xK_s^1(\xH_N)$:
    \begin{align}
        -i[\Hamil,\boldsymbol{\sigma}_N]=-i[\Hamil_{N+d},\boldsymbol{\sigma}_N].
    \end{align}
\end{corollary}
\begin{proof}
    Using \cref{proposition:polyink}, $\Hamil\boldsymbol{\sigma}_N=\Proj{N+d}\Hamil\boldsymbol{\sigma}_N$. Besides, $\Proj{N+d}\boldsymbol{\sigma}_N=\boldsymbol{\sigma}_N$, so that $\Hamil\boldsymbol{\sigma}_N= \Hamil_{N+d}\boldsymbol{\sigma}_N$.
\end{proof}
\begin{remark}

    We recall that as $\boldsymbol{\sigma}_N\in \xK_s^1(\xH_N)$ and $\xH_N\subset \xH$, one has $\boldsymbol{\sigma}_N\in \xK_s^1(\xH)$.
\end{remark}
\begin{corollary}
    \label{co:dissipator}
    Let $\LL$ be a polynomial operator of degree $d$ in $\destroy$ and $\create$. Then, for every density operator $\boldsymbol{\sigma}_N\in \xK_s^1(\xH_N)$,
    \begin{align}
        \cD_{\LL}(\boldsymbol{\sigma}_N)=\cD_{\LL_{N+2d}}(\boldsymbol{\sigma}_N).
    \end{align}
\end{corollary}
\begin{proof}
    Using \cref{proposition:polyink}, one has $\LL \boldsymbol{\sigma}_N \LL^\dag =\LL_{N+d}\boldsymbol{\sigma}_N \LL_{N+d}^\dag$. In general $\LL^\dag \LL \boldsymbol{\sigma}_N$ is not equal to $\LL_{N+d}^\dag \LL_{N+d}\boldsymbol{\sigma}_N$ (a counter-example is given in \cref{ex:squeezecat}). Nevertheless, we have $\LL^\dag \LL \boldsymbol{\sigma}_N=\LL_{N+2d}^\dag \LL_{N+2d}\boldsymbol{\sigma}_N$, which ensures that $\cD_{\LL}(\boldsymbol{\sigma}_N)=\cD_{\LL_{N+2d}}(\boldsymbol{\sigma}_N)$.
\end{proof}

Extension to several bosonic modes is rather straightforward. Let ${Q}[\mathbf{X_1},\mathbf{Y_1},\ldots \mathbf{X_m},\mathbf{Y_m}]$ be a polynomial of degree $d$. We easily generalize \cref{proposition:polyink} to $m$ modes and obtain
\begin{align}
    {Q}[\destroy_1,\destroy^\dag_1,\ldots, \destroy_m,\destroy^\dag_m]\xH_{N_1,\ldots N_m} \subset \xH_{N_1+d,\ldots, N_m+d}.
\end{align}
As a consequence, $\xL(\boldsymbol{\sigma}_N)$ coincides with $\xL_{N+d}(\boldsymbol{\sigma}_N)$, where $d=\max(d_{\Hamil}, 2\max_i(d_{i}))$ with $d_{\Hamil}$ the degree of a polynomial generating the Hamiltonian, $d_i$ the degree of $\LL_i$, and $N+d$ denotes the m-uplet $(N_1+d,\ldots N_m+d)$.

\subsection{Some examples}
We showed that it is always possible to compute the space estimate of \cref{sec:sec_1} expressing $\xL_{N+d}$, $\xL_N$ and $\boldsymbol{\rho}_N$ on $\xK_ s^1(\xH_{N+d})$, and neglecting the time-discretization errors.
In this part, we investigate some explicit examples to provide an intuition on the estimates. In \cref{ex:cat,ex:catplusbuffer}, we also compare our estimate against the truncation error to assess the degree of overestimation in the upper bound. We also assume that $\boldsymbol{\rho}_0\in \xK^1_+(\xH_N)$ in all these examples.
\subsubsection{\texorpdfstring{Frequency driven damped oscillator - $\Hamil = u(t) \create \destroy$ and $\cD_{\destroy}$}{Frequency driven damped oscillator}}
\label{ex:aadag}
The Hamiltonian $u(t) \create \destroy$ for a scalar function $u(t)$ is not responsible for any truncation error as $\Proj{N}$ commutes with $\Nop=\create \destroy$ ($\Nop$ being diagonal in the Fock basis). As a consequence, $i[\Hamil-\Hamil_N,\rN]=0$.

Concerning the dissipator, it can be noticed that $\destroy \rN = \Proj{N} \destroy \Proj{N} \rN$ and $\rN \create = \rN \Proj{N} \create \Proj{N} $. Besides, $\create_N \destroy_N \rN=\create \destroy \rN$, so that $\cD_{\destroy}(\rN)=\cD_{\destroy_N}(\rN)$.
As a consequence, $\xL_N(\boldsymbol{\sigma}_N)=\xL(\boldsymbol{\sigma}_N)$ for every $\boldsymbol{\sigma}_N\in \xK^1_s(\xH_N)$ and the truncation is not responsible for any error. Both the truncation error and our estimate are null.

\subsubsection{\texorpdfstring{Driven oscillator - $\Hamil = u(t)(\destroy + \create)$}{Driven oscillator}}
\label{ex:aplusadag}
We have
\begin{align}
    \notag i[\Hamil - \Hamil_N,\boldsymbol{\rho}_N]&=i\Proj{N}^\perp \Hamil \rN +h.c.\\
    \notag &=i\left( \sum_{n=N}^\infty \sqrt{n+1}\ket{n+1}\bra{n} + \sum_{n=N+1}^\infty\sqrt{n+1}\ket{n}\bra{n+1}\right)\boldsymbol{\rho}_N + h.c.\\
    &=i\sqrt{N+1}\ket{N+1}\bra{N}\boldsymbol{\rho}_N + h.c.\ .
\end{align}
Hence, we have to compute
\begin{multline}
    \|[\Hamil-\Hamil_N,\rN(t)]\|_1=|u(t)| \\
    \times\xtr{\sqrt{\left(i\sqrt{N\!+\!1}\ket{N\!+\!1}\bra{N}\rN(t) + h.c.\right)\left(i\sqrt{N\!+\!1}\ket{N\!+\!1}\bra{N}\rN(t) + h.c.\right)^{\dag}}}.
\end{multline}
Using $\rN \ket{N+1}=0$, we have
\begin{align}
    \notag &\left(i\sqrt{N+1}\ket{N+1}\bra{N}\rN(t) + h.c.\right)\left(i\sqrt{N+1}\ket{N+1}\bra{N}\rN(t) + h.c.\right)^{\dag}\\
    &=\rN\ket{N}\bra{N}\rN + \ket{N+1}\bra{N}\rN^2\ket{N}\bra{N+1}.
\end{align}
As $\rN \ket{N}$ and $\ket{N+1}$ are orthogonal vectors, we get
\begin{multline}
    \sqrt{\rN\ket{N}\bra{N}\rN + \ket{N\!+\!1}\bra{N}\rN^2\ket{N}\bra{N\!+\!1}}\\
    = \sqrt{\rN\ket{N}\bra{N}\rN} + \sqrt{\ket{N\!+\!1}\bra{N}\rN^2\ket{N}\bra{N\!+\!1}}.
\end{multline}
Then, as for a rank one symmetric matrix $S$, $\xtr{\sqrt{S}}=\sqrt{\xtr{S}}$, we obtain
\begin{align}
    \notag \|[\Hamil-\Hamil_N,\rN(t)]\|_1=&=|u(t)|\sqrt{N+1} \Bigl( \sqrt{\xtr{
    \rN\ket{N}\bra{N}\rN}} + \sqrt{\bra{N}{\rN(t)}^2 \ket{N}}\Bigr)\\
    &=2|u(t)|\sqrt{N+1} \sqrt{\bra{N}{\rN(t)}^2 \ket{N}}.
\end{align}
From a numerical point of view, the estimate is computationally cheap to compute, as one only needs to compute the Hilbert norm of the N$^{th}$ row of $\boldsymbol{\rho}_N$. Hence, the space estimate obtained in \cref{prop-aposteriori} gives
\begin{align}
    \|\boldsymbol{\rho}(t)-\rN(t)\|_1 \leq \int_0^t 2|u(s)|\sqrt{N+1} \sqrt{\bra{N}{\rN(s)}^2 \ket{N}}ds.
\end{align}
\subsubsection{\texorpdfstring{Dissipative cat-qubit - $\LL = \destroy^2- \alpha^2 \Id$, with $\alpha \in \mathbb{R}$}{Dissipative cat-qubit}}
\label{ex:cat}
The dissipator $\destroy^2- \alpha^2 \Id$ is used for the stabilization of dissipative cat-qubits, see \cite{mirrahimiDynamicallyProtectedCatqubits2014}. Well-posedness and convergence toward the codespace $\operatorname{Span} \{ \ket{\pm \alpha}\bra{\pm \alpha},\ket{\pm \alpha}\bra{\mp \alpha}  \}$ is proved in \cite{azouitWellposednessConvergenceLindblad2016}.
\paragraph{Expression of the estimate}

One starts considering 
\begin{align}
    \notag \cD_{\LL}(\rN(t))-\cD_{\LL_N}(\rN(t))&=  -\frac{\alpha^2}{2}  \big(\sqrt{(N\!+\!1)(N\!+\!2)} \notag\\
    &\quad\times( \ket{N\!+\!2}\bra{N} \rN(t)+ \rN(t) \ket{N}\bra{N\!+\!2})\\
    &+
    \sqrt{N(N\!+\!1)} ( \ket{N\!+\!1}\bra{N\!-\!1} \rN(t)\notag\\
    &\quad+ \rN(t) \ket{N\!-\!1}\bra{N\!+\!1}) \big),
\end{align}
which is an operator supported on $\operatorname*{Span}\{ \rN \ket{N}, \rN \ket{N-1}, \ket{N+1},\ket{N+2}\}$. After some computations, postponed to \cref{eq:cat}, we get
\begin{align}
    \notag \|\cD_{\LL}(\rN)-\cD_{\LL_N}(\rN)\|_1 &= \sqrt{N\!+\!1}\frac{\alpha^2}{2} \notag\\
    &\quad\times\Biggl(\xtr{\sqrt{\rN \Big(\sqrt{N}\ket{N\!-\!1}\bra{N\!-\!1} + \sqrt{N\!+\!2}\ket{N}\bra{N}\Big)\rN}}\\
    & + \sqrt{N}\bra{N\!-\!1}\rN^2\ket{N\!-\!1} + \sqrt{N\!+\!2}\bra{N}\rN^2\ket{N} \Biggr).
\end{align}

Note that this expression is not expensive to numerically compute, as it only requires  performing elementary operations on the last two rows of $\rN$ and computing the eigenvalues of a rank 2 matrix.  

\begin{figure}[hbt]
    \includegraphics[width=0.48\textwidth]{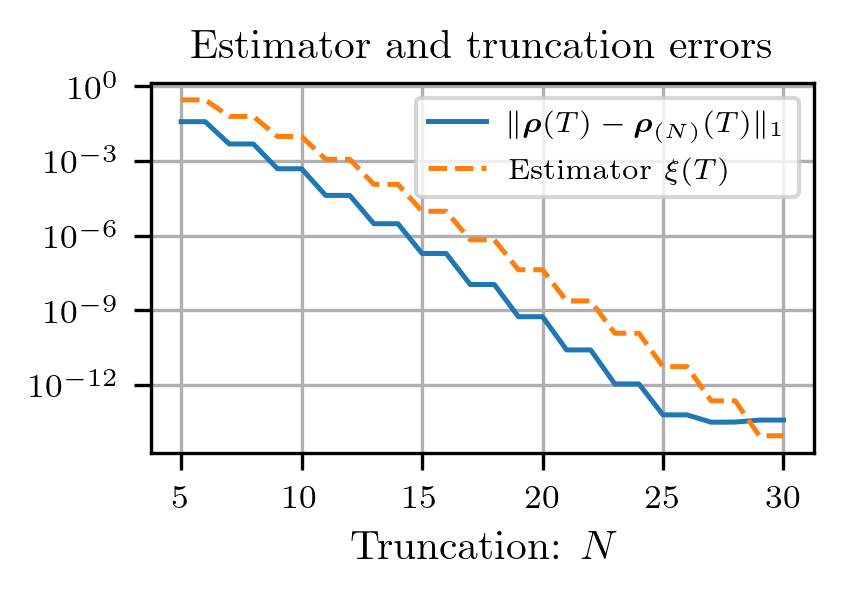}
    \hfill
    \includegraphics[width=0.48\textwidth]{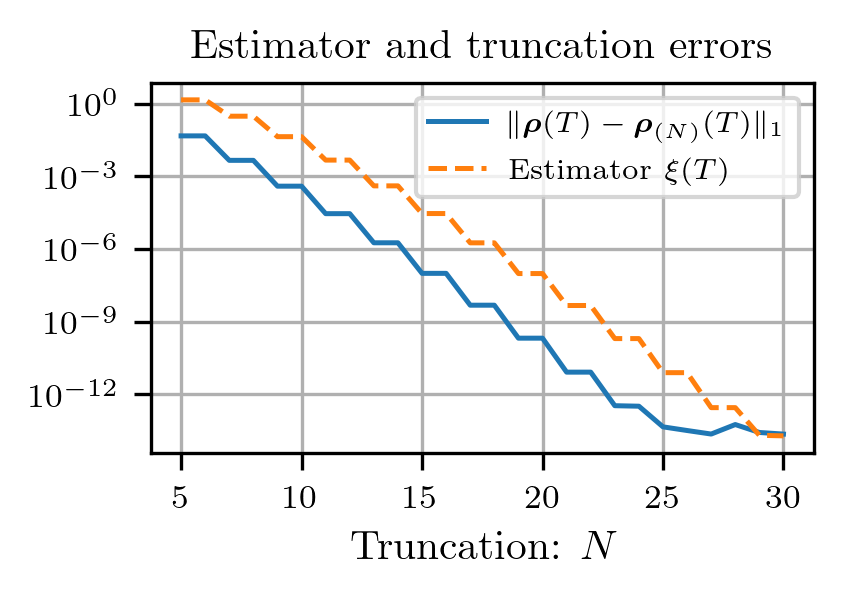}
    \caption{Left plot: Dissipative cat-qubit - $\LL = \destroy^2- \alpha^2 \Id$, $\alpha=1$. Right plot: Dissipative squeezed-cat qubit - $\LL = (\cosh(r)\destroy + \sinh(r)\create)^2 - \alpha^2 \Id$, $\alpha=1$, $r=5/4$. In both cases, $\|\boldsymbol{\rho}(T) - \rN(T)\|_1$ is the truncation error at the end of the simulation described in \cref{ex:cat,ex:squeezecat}, and $\xi(T)$ is the estimate output. Errors and estimates change only at odd numbers due to the preservation of the parity operator $e^{i \pi \create \destroy}$. Saturation occurs below $10^{-13}$ due to the precision of the time solvers.}
    \label{fig:estimator_1D}
\end{figure}

\paragraph{Numerical test}

We simulate the dynamics for various truncations indexed by $N$ in \cref{fig:estimator_1D}, starting with an initial state $\boldsymbol{\rho}_0 = \ket{0}\bra{0}$. Our focus is on evaluating the estimate's performance at time $T = 1$ and comparing it to the exact error $\|r(T)\|_1 = \|\boldsymbol{\rho}(T) - \rN(T)\|_1$. The numerical solution is obtained using an adaptive high-order Runge-Kutta method, with absolute and relative tolerances set below $10^{-14}$. The reference solution is computed with $N = 40$. The estimate for the reference yields an error below $ 4 \cdot 10^{-15}$, indicating that, aside from errors introduced by the time solver and finite numerical precision (both estimated to be on the order of $10^{-13}$), we have a certification that the computed $\|r(T)\|_1$ is accurate to within $4 \cdot 10^{-15}$. Indeed, we have 
\begin{align}
    \notag &\|\rN(T)-\boldsymbol{\rho}_{(40)}(T) \|_1 - \|\boldsymbol{\rho}_{(40)}(T)-\boldsymbol{\rho}(T)\|_1\notag\\
    &\quad\leq\|r(T)\|_1 \leq \|\rN(T)-\boldsymbol{\rho}_{(40)}(T) \|_1 + \|\boldsymbol{\rho}_{(40)}(T)-\boldsymbol{\rho}(T)\|_1, \\
    &\|\boldsymbol{\rho}_{(40)}(T)-\boldsymbol{\rho}(T)\|_1 \leq \xi_{40}(T) \leq 4 \cdot 10^{-15}.
\end{align}

\subsubsection{\texorpdfstring{Dissipative squeezed-cat qubit  - $\LL = (cosh(r)\destroy + sinh(r)\create)^2 - \alpha^2 \Id$, with $\alpha, r \in \mathbb{R}$}{Dissipative squeezed-cat qubit }}
\label{ex:squeezecat}
This dissipator generalizes the previous one (corresponding to $r=0$) used to stabilize squeezed cat states~\cite{Hillmann_2023}. Note that for $r>0$, $\xL(\rN)$ belongs to $\xK^1_s(\xH_{N+4})$ rather than just $\xK^1_s(\xH_{N+2})$.
We repeat the simulations from \cref{ex:cat} with $r=5/4$ and ensure the same accuracy on the reference. Numerical results are reported in \cref{fig:estimator_1D}.

\subsubsection{\texorpdfstring{Dissipative cat-qubit with buffer - $\Hamil = (\destroy^2 - \alpha^2 \Id)\createb + (\create^2 - \alpha^2 \Id)\destroyb$, with $\alpha\in\mathbb{R}$ and $\cD_{\destroyb}$}{Dissipative cat-qubit with buffer}}
\label{ex:catplusbuffer}
This two-modes bosonic system describes the dissipative engineering of two-photon loss for a dissipative cat qubit using a lossy buffer cavity. For physical motivation, we refer to the pioneering article \cite{mirrahimiDynamicallyProtectedCatqubits2014}. Well-posedness and convergence toward the codespace $\operatorname{Span} \{ \ket{\pm \alpha}\bra{\pm \alpha}\otimes \ket{0}\bra{0},\ket{\pm \alpha}\bra{\mp \alpha} \otimes \ket{0}\bra{0} \}$ is proved in \cite{robinConvergenceBipartiteOpen2023}.

As in \cref{ex:aadag}, the dissipator $\cD_{\destroyb}$ does not induce errors, meaning it is enough to focus on the Hamiltonian part. We recall that $\xH_{(n_1,n_2)}=\{\ket{i}\otimes \ket{j} \mid i\leq n_1,j \leq n_2\}$. We can then easily check that $\Hamil\rN\in \xK^1(\xH_{(n_1+2,n_2+1)})$ and that $\Hamil_{(n_1+2,n_2+1)}\rN=\Hamil\rN$. While an explicit expression can be obtained, in practice we simply compute $\Hamil_{(n_1+2,n_2+1)}-\Hamil_{(n_1,n_2)}\in \mathcal{O}(\xH_{(n_1+2,n_2+1)})$, and compute the trace norm (in $\xK^1_s(\xH_{(n_1+2,n_2+1)})$) of $\| (\Hamil_{(n_1+2,n_2+1)}-\Hamil_{(n_1,n_2)}) \rN\|_1$.

\begin{figure}[hbt]
    \centering
    \includegraphics[width=1.0\textwidth]{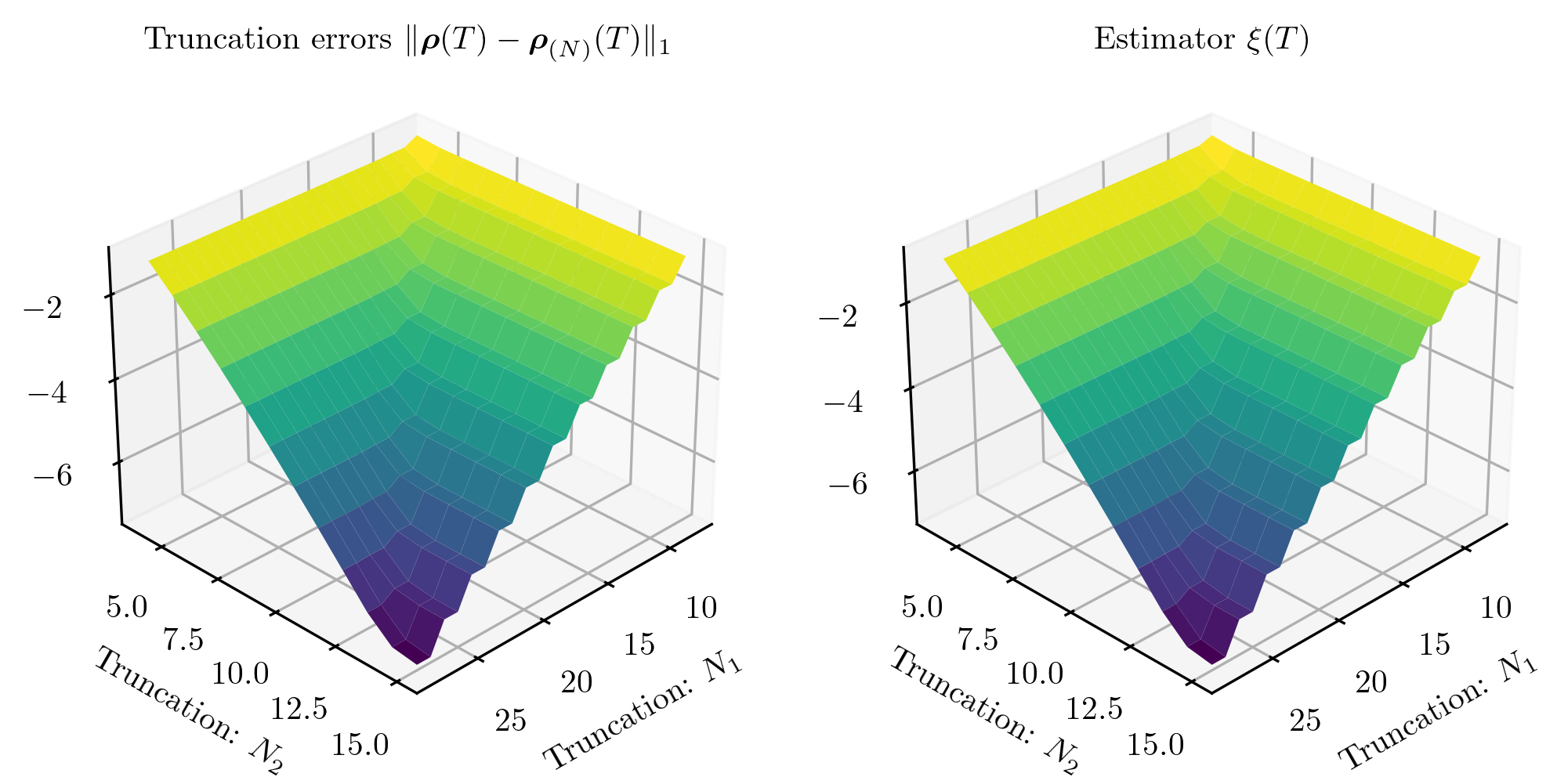}
    \captionsetup{justification=centering}
    \caption{3D plots in log scale of the truncation error $\|\boldsymbol{\rho}(T) - \rN(T)\|_1$ (\textit{left}) and the estimate $\xi(T)$ (\textit{right}). Slices of these plots are reproduced in \cref{fig:estimator_2D_slice}.}
    \label{fig:estimator_2D}
\end{figure}

\begin{figure}[!htbp]
    \centering
    \begin{minipage}[b]{0.49\textwidth}
        \centering
        \includegraphics[width=\linewidth]{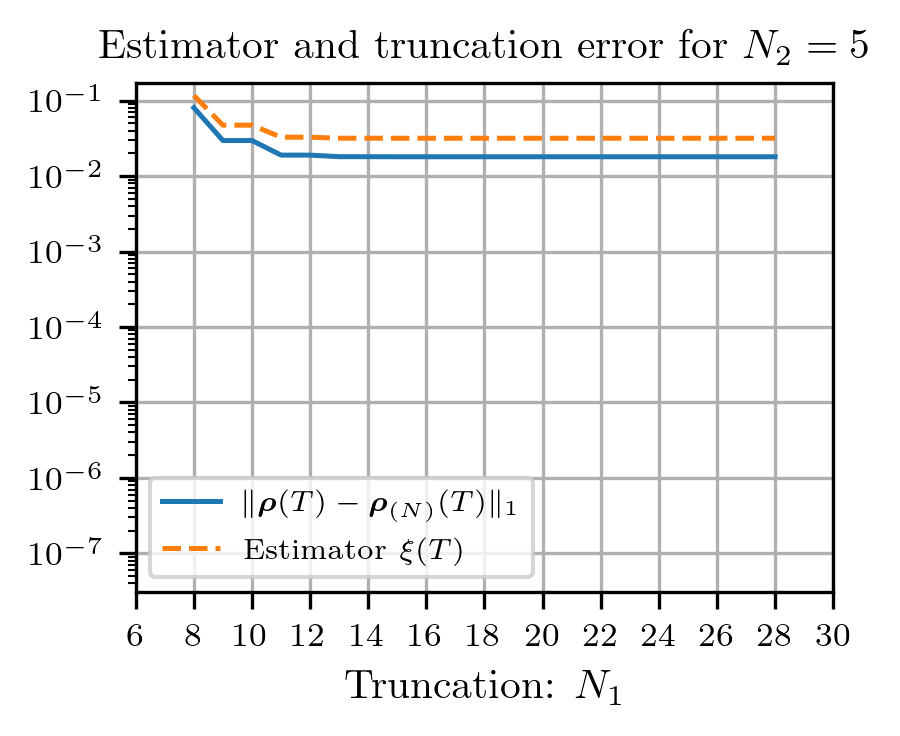}
    \end{minipage}
    \begin{minipage}[b]{0.49\textwidth}
        \centering
        \includegraphics[width=\linewidth]{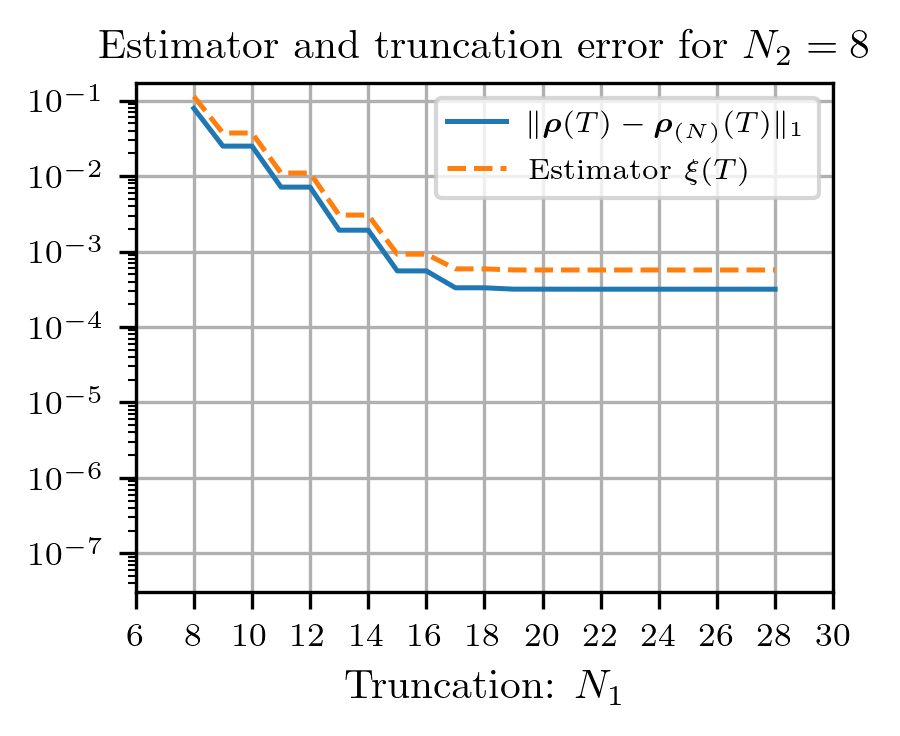}
    \end{minipage}

    \begin{minipage}[b]{0.49\textwidth}
        \centering
        \includegraphics[width=\linewidth]{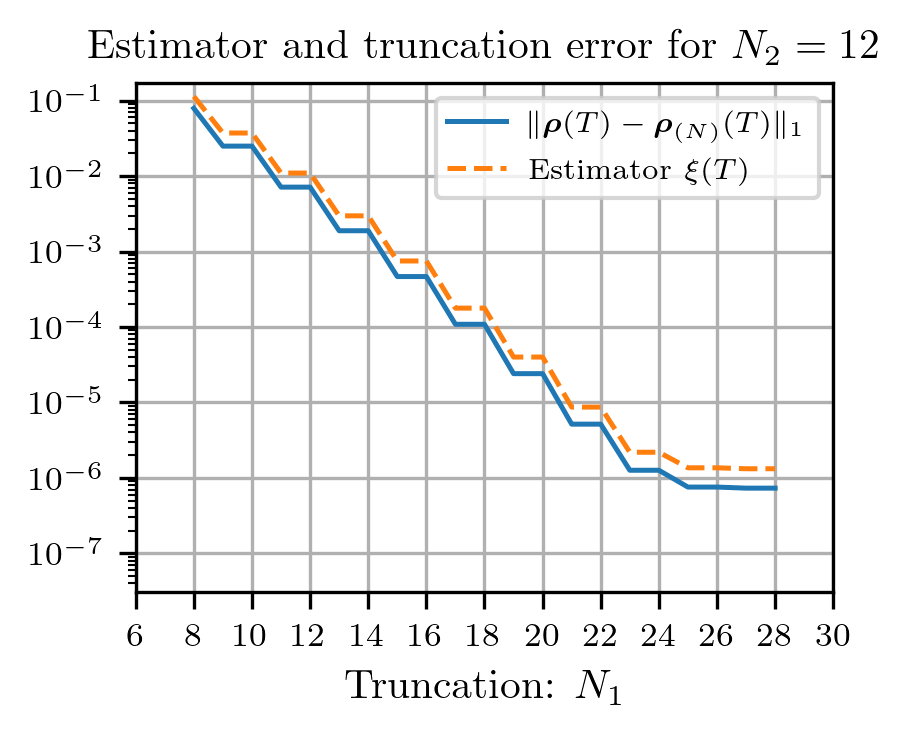}
    \end{minipage}
    \begin{minipage}[b]{0.49\textwidth}
        \centering
        \includegraphics[width=\linewidth]{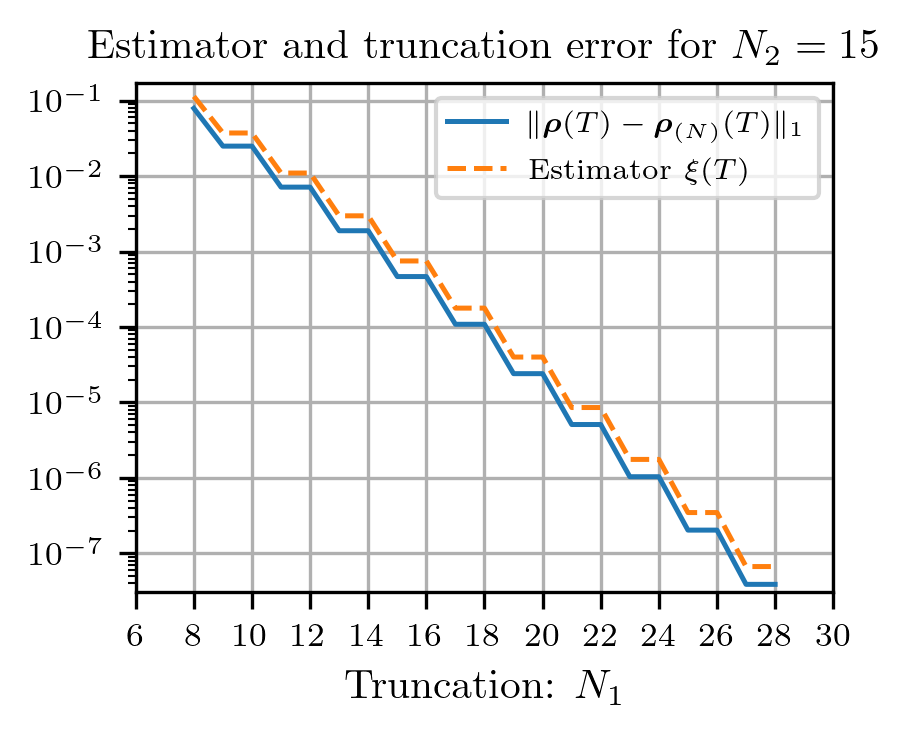}
    \end{minipage}
    
    \caption{Slices of the 3D plot in \cref{fig:estimator_2D} with fixed $N_2$.}
    \label{fig:estimator_2D_slice}
\end{figure}
\paragraph{Numerical test}

We simulate the dynamics for various truncations (from ($N_1 = 8, N_2 = 4$) to ($N_1 = 28, N_2 = 15$)) in \cref{fig:estimator_2D}, starting with an initial state $\boldsymbol{\rho}_0 = \ket{00}\bra{00}$ and setting $\alpha=1$. Our focus is on evaluating the estimate's performance at time $T = 1$ and comparing it to the error $\|r(T)\| = \|\boldsymbol{\rho}(T) - \rN(T)\|_1$. The numerical solution is obtained using an adaptive high-order Runge-Kutta method, with absolute and relative tolerances set below $10^{-14}$. The reference solution is computed with $(n_1 = 40, n_2 = 20)$. The estimate for the reference yields an error below $3 \cdot 10^{-15}$, indicating that, aside from errors introduced by the time solver and finite numerical precision (both estimated to be on the order of $10^{-13}$), we have certification that the computed $\|r(T)\|$ is accurate to within $3 \cdot 10^{-15}$. \cref{fig:estimator_2D_slice} provides  plots of some slices of \cref{fig:estimator_2D} for fixed $N_2$.

\section{Application 2: estimates for Lindbladian with unitary operators involved}
\label{sec_ex2}

\subsection{Unitary operators}
\label{subsec:unitary}

In this section, we compute the \textit{a posteriori} truncation estimate given in \cref{prop-aposteriori} for Lindbladians that involve either a unitary operator as Hamiltonian or more interestingly unitary operators as jump operators.

We are interested in operators for which we know the exact finite truncation. More precisely, we assume that analytic formulas for the coefficients of $\Uop_N=\Proj{N}\Uop\Proj{N}$ are available. For example, the matrix elements $\bra{i}e^{i\eta X}\ket{j}$, for any $\eta \in \mathbb{R}$, and $(i,j)\in \xN^2$, admit an explicit (and numerically stable) formula, see e.g. \cite[Appendix E.3]{PhysRevX.15.011011}. Hence, we have access to the truncated operator $\Uop_N$ exactly. The key estimate for this section is the following Lemma.
\begin{lemma}
    \label{lem_unitary}
    Let $\Uop$ be a unitary operator on $\xH$, and $\mathbf{M}_N$ be an operator on $\xH_N$. We have the following norm equality:
    \begin{equation}
        \label{th:unitary}
        \|\Proj{N}^\perp \Uop \mathbf{M}_N\|_1 = \xtr{\sqrt{\mathbf{M}_N^\dagger(\Id_N - \Uop_N^\dagger \Uop_N)\mathbf{M}_N}}.
    \end{equation}
\end{lemma}
\begin{proof}
    We write $\Uop$ using the space decomposition $\xH=\xH_N \oplus \xH_N^\perp$:
    \begin{align}
        \Uop= \begin{pmatrix}
            \Uop_N&\Proj{N} \Uop \Proj{N}^\perp \\
            \Proj{N}^\perp \Uop \Proj{N}&\Proj{N}^\perp \Uop \Proj{N}^\perp 
        \end{pmatrix}.
    \end{align}
    Since $\Uop$ is unitary, $\Proj{N}( \Uop^\dag \Uop )\Proj{N}=\Id_N=\Uop_N^\dagger \Uop_N + (\Proj{N}^\perp \Uop \Proj{N})^\dagger (\Proj{N}^\perp \Uop \Proj{N})$. Hence, $(\Proj{N}^\perp \Uop \Proj{N})^\dagger (\Proj{N}^\perp U \Proj{N}) = \Id_N - \Uop_N^\dagger \Uop_N$.
    As a consequence,
    \begin{align}
        \notag
        \|\Proj{N}^\perp \Uop \mathbf{M}_N\|_1 &= \xtr{\sqrt{\mathbf{M}_N^\dagger (\Proj{N}^\perp \Uop \Proj{N})^\dagger (\Proj{N}^\perp \Uop \Proj{N}) \mathbf{M}_N}} \\
        \label{eq_sqrt}&= \xtr{\sqrt{\mathbf{M}_N^\dagger (\Id_N - \Uop_N^\dagger \Uop_N) \mathbf{M}_N}}.
    \end{align}
\end{proof}
Note that the symmetric operator inside the square root of \cref{eq_sqrt} can be computed and has support on $\xH_N$.

\paragraph{Hamiltonian error involving a unitary operator}
While being quite specific, let us assume that $\Hamil$ is also unitary. In this case
\begin{align}
    \|[\Hamil-\Hamil_N,\rN]\|_1 \leq 2 \|(\Hamil-\Hamil_N)\rN\|_1 = 2 \|\Proj{N}^\perp \Hamil \rN\|_1.
\end{align}
As a consequence, we can use \cref{lem_unitary} to compute the error.
\paragraph{Dissipator error involving a unitary operator}\mbox{}\\
We split into three parts the dissipator's error $(\cD_{\Uop}-\cD_{\Uop_N})(\rN)$, with $\Uop$ a unitary operator on $\xH$.
\begin{align}
    \label{eq_decompoDU}
    \notag
    \|(\cD_{\Uop}-\cD_{\Uop_N})(\rN)\|_1 &\leq \|\Uop \rN \Uop^\dag-\Uop_N \rN \Uop_N^\dag\|_1 \notag\\
    &+ \|\Uop^\dag \Uop \rN- \Uop_N^\dag \Uop_N \rN\|_1 \\
    &+ \|\rN \Uop^\dag \Uop-\rN  \Uop_N^\dag \Uop_N \|_1.
\end{align}
We start by the second and third term of the right-hand side of \cref{eq_decompoDU}.
\begin{align}
    \notag \|\rN \Uop^\dag \Uop-\rN  \Uop_N^\dag \Uop_N \|_1&=\|\Uop^\dag \Uop \rN- \Uop_N^\dag \Uop_N \rN\|_1\notag\\
    &=\|\rN- \Uop_N^\dag \Uop_N \rN\|_1\\
    &= \|(\Id_N-\Uop_N^\dag \Uop_N)\rN\|_1.
\end{align}
This expression can be numerically computed. Focusing on the remaining term, we get
\begin{equation}
    \|\Uop \rN \Uop^\dag - \Uop_N \rN \Uop_N^\dag\|_1 \leq \| \Proj{N}^\perp \Uop \rN \Uop^\dag \Proj{N}^\perp\|_1 + \| \Proj{N} \Uop \rN \Uop^\dag \Proj{N}^\perp\|_1 + \| \Proj{N}^\perp \Uop \rN \Uop^\dag \Proj{N}\|_1.
\end{equation}
The last two terms are equal and can be handled using \cref{th:unitary} with $\mathbf{M}_N=\rN \Uop_N^\dagger$:
\begin{align*}
    \notag \| \Proj{N}^\perp \Uop \rN \Uop^\dag \Proj{N}\|_1 &= \| \Proj{N}^\perp \Uop \rN \Uop^\dag_N\|_1\\
    &= \xtr{\sqrt{\Uop_N \rN \Uop_N^\dagger (\Id_N - \Uop_N^\dagger \Uop_N) \Uop_N \rN \Uop_N^\dagger}}.
\end{align*}
The remaining term is $\| \Proj{N}^\perp \Uop \rN \Uop^\dag \Proj{N}^\perp\|_1$. As $\rN \geq 0$, we also have $\Proj{N}^\perp \Uop \rN \Uop^\dag \Proj{N}^\perp \geq 0$. Hence, the trace norm and the trace of these operators coincide. Then,
\begin{align}
    \notag \xtr{\Proj{N}^\perp \Uop \rN \Uop^\dagger \Proj{N}^\perp}&=\xtr{\Uop \rN \Uop^\dag \Proj{N}^\perp}\\
    \notag &=\xtr{\Uop \rN \Uop^\dag (\Id-\Proj{N})}\\
    &= 1 - \xtr{\Proj{N} \Uop \rN \Uop^\dag \Proj{N}}.
\end{align}
Eventually, we get the following estimate, that can be numerically computed: 
\begin{align}
    \|(\cD_{\Uop}-\cD_{\Uop_N})(\rN)\|_1 &\leq 2\|(\Id_N-\Uop_N^\dag \Uop_N)\rN\|_1 + \| \Proj{N}^\perp \Uop \rN \Uop^\dag_N\|_1 + 1 - \xtr{ \Uop_N \rN \Uop_N^\dag }.
\end{align}
\subsection{Application to a more complex example}
\label{subsec:GKP}
In this section, we study a more complex example involving both unitary operators and polynomials in creation and annihilation operators. More precisely, we study the following dynamics introduced in \cite{PhysRevX.15.011011} for dissipative GKP stabilization:
\begin{equation}
	\label{eq:lindblad-GKP}
	\frac{d}{dt} \boldsymbol{\rho} = \sum_{k=0}^{3} \cD_{\LL_k} (\boldsymbol{\rho}),
\end{equation}
with
\begin{align}
    \LL_0 = \cA \, e^{i\eta\qop} \left(  \Id-\epsilon\Pop\right) -  \Id,\quad
	\mathbf{R}= e^{i\pi {\create \destroy}/2}, \quad
	\LL_k = \mathbf{R}^k \, \LL_0 \, \mathbf{R}^{-k},
    \label{eq:def-LGKP}
\end{align}
where $\mathcal{A},\eta$ and $\epsilon$ are given constant real numbers. We recall that $\qop$ and $\Pop$ are the position and momentum operators whose definition is recalled in \cref{subsec_notation}. Note also that $\mathbf{R}$ is unitary and $\mathbf{R}^4=\Id$. To lighten a little the notations, we introduce $\Uop=e^{i\eta \qop}$, and $\xP=\cA \left(  \Id-\epsilon\Pop\right)$.

Our goal is obtaining a good upper bound on $\|\xL(\rN)-\xL_N(\rN)\|_1$ that can be numerically computed. Because the dissipator $\LL_0$ involves both the polynomial $\xP$ and the unitary operator $e^{i\eta\qop}$, we cannot directly apply the results of \cref{sec_ex1} or \cref{subsec:unitary}.

\paragraph{Tools}
Let us first state a generalization of \cref{th:unitary}. Assume $\xH_{N_1}\subset \xH_{N_2}$, and let us consider $\Uop$ a unitary operator on $\xH$, and $\mathbf{M}_{N_2}\in \mathcal{O}(\xH_{N_2})$.  Note that $\Proj{N_1}^\perp = \Id-\Proj{N_1}=\Id-\Proj{N_2}+\Proj{N_2}-\Proj{N_1}$. Besides as $\Proj{N_2}^\perp$ and $\Proj{N_2}-\Proj{N_1}$ are orthogonal projectors with orthogonal ranges, we have
\begin{align}
	\notag\|\Proj{N_1}^\perp \Uop \mathbf{M}_{N_2}\|_1 &=  \|\Proj{N_2}^\perp \Uop \mathbf{M}_{N_2}\|_1 + \|(\Proj{N_2} - \Proj{N_1}) \Uop \mathbf{M}_{N_2}\|_1.
\end{align}
Besides, using \cref{th:unitary}, we get 
\begin{align}
	\label{th:unitary_mod}
	\|\Proj{N_1}^\perp \Uop \mathbf{M}_{N_2}\|_1&=  \xtr{\sqrt{\mathbf{M}_{N_2}^\dagger(\ \Id_{N_2} - \Uop_{N_2}^\dagger \Uop_{N_2})\mathbf{M}_{N_2}}} + \|(\Proj{N_2} - \Proj{N_1}) \Uop \mathbf{M}_{N_2}\|_1.
\end{align}

\paragraph{Dissipator error}\mbox{}\\
Let us start the analysis with the error term coming from $\LL_0$.
\begin{align}
    (\cD_{\Uop\xP- \Id}-\cD_{\Proj{N} \Uop \xP \Proj{N} -  \Id_N})(\rN)&= \operatorname{\mathbf{I}}_1-\frac{1}{2}(\operatorname{\mathbf{I}}_2 + \operatorname{\mathbf{I}}_2^\dag),
\end{align}
with
\begin{align}
	\operatorname{\mathbf{I}}_1&=(\Uop \xP- \Id) \rN (\xP^\dagger \Uop^\dag- \Id) - \Proj{N}(\Uop \xP- \Id) \rN (\xP^\dagger \Uop^\dag- \Id)\Proj{N},\\
	\operatorname{\mathbf{I}}_2&=(\xP^\dagger \Uop^\dag- \Id) (\Uop \xP- \Id) \rN- \Proj{N}(\xP^\dagger \Uop^\dag- \Id) \Proj{N} (\Uop \xP- \Id) \rN.
\end{align}
We start with $\operatorname{\mathbf{I}}_1$:
\begin{align}
	\notag \operatorname{\mathbf{I}}_1 &= (\Uop \xP \rN \xP^\dagger \Uop^\dag)  + \rN - (\rN \xP^\dagger \Uop^\dag) - (\Uop \xP \rN)\\
	\notag& \quad - \Bigl[(\Proj{N} \Uop \xP \rN \xP^\dagger \Uop^\dag \Proj{N}) + \rN - (\rN \xP^\dagger \Uop^\dag \Proj{N}) - (\Proj{N} \Uop \xP \rN)\Bigr]\\
	 & = \Uop \xP \rN \xP^\dagger \Uop^\dag - \Proj{N} \Uop \xP \rN \xP^\dagger \Uop^\dag \Proj{N} -  (\Proj{N}^\perp \Uop \xP \rN + h.c.).
\end{align}
Similarly,
\begin{align}
	\label{GKP_term2}
	\notag \operatorname{\mathbf{I}}_2&=(\xP^\dagger \Uop^\dag -  \Id)(U\xP -  \Id)\rN - \Proj{N}(\xP^\dagger \Uop^\dag -  \Id)\Proj{N}(\Uop \xP -  \Id)\rN\ \\
	&=(\xP^\dagger\xP - \Proj{N}\xP^\dagger \Uop_{N+1}^\dagger \Proj{N} \Uop_{N+1} \xP \Proj{N})\rN -\Proj{N}^\perp \Uop \xP \rN - \Proj{N}^\perp \xP^\dagger \Uop^\dag \rN.
\end{align}
Thus,
\begin{align}
	\operatorname{\mathbf{I}}_1-\frac{1}{2}(\operatorname{\mathbf{I}}_2 + \operatorname{\mathbf{I}}_2^\dag)&=\mathbf{A}_1+\mathbf{A}_2+\mathbf{A}_3+\mathbf{A}_4,
\end{align}
with
\begin{align}
	\mathbf{A}_1&=\Uop \xP \rN \xP^\dagger \Uop^\dag - \Proj{N} \Uop \xP \rN \xP^\dagger \Uop^\dag \Proj{N} ,\\
	\mathbf{A}_2&=-(\xP^\dagger\xP - \Proj{N}\xP^\dagger \Uop_{N+1}^\dagger \Proj{N} \Uop_{N+1} \xP \Proj{N})\rN,\\
	\mathbf{A}_3&=\frac{1}{2}\Proj{N}^\perp \Uop \xP \rN + h.c.\, ,\\
	\mathbf{A}_4&=\frac{1}{2} \Proj{N}^\perp \xP^\dagger \Uop^\dag \rN + h.c.\, 
\end{align}
We then apply the triangle inequality. It remains to find a way to compute the trace norm of each of the $A_i$. 
First, we get
\begin{align}
	\label{eq_30} &
	\|\mathbf{A}_1\|_1 \leq \| \Proj{N}^\perp \Uop \xP \rN \xP^\dagger \Uop^\dag \Proj{N}^\perp\|_1 + 2\| \Proj{N}^\perp \Uop \xP \rN \xP^\dagger \Uop^\dag \Proj{N}\|_1.
\end{align}
The second term of the right-hand side of \cref{eq_30} can be numerically computed using \cref{th:unitary_mod}, with $\mathbf{M}_{N_2}$ being  $\xP \rN \xP^\dagger \Uop^\dag \Proj{N}\in \mathcal{O}(\xH_{N+1})$. The first term is handled as follows:
\begin{align}
	\notag \| \Proj{N}^\perp \Uop \xP \rN \xP^\dagger \Uop^\dag \Proj{N}^\perp\|_1&=\xtr{\Proj{N}^\perp \Uop \xP \rN \xP^\dagger \Uop^\dag \Proj{N}^\perp}\\
	\notag &=\xtr{\Proj{N}^\perp \Uop \xP \rN \xP^\dagger \Uop^\dag}\\
	\notag &=\xtr{\xP \rN \xP^\dagger}- \xtr{\Proj{N} \Uop \xP \rN \xP^\dagger \Uop^\dag \Proj{N}}.
\end{align}
Then, $\mathbf{A}_2$ and $\mathbf{A}_3$ can be numerically computed as an element of resp. $\mathcal{O}(\xH_{N+2})$ and $\mathcal{O}(\xH_{N+1})$. For $\mathbf{A}_4$, we use the Heisenberg-picture evolution of $\Pop$ under the Hamiltonian $\qop$ to get
\begin{align}
	\notag \xP^\dag \Uop^\dag &=\cA(\Id-\epsilon \Pop) e^{-i\eta \qop}\\
	\notag &=\cA e^{-i\eta \qop}e^{i\eta \qop}(\Id-\epsilon \Pop) e^{-i\eta \qop}\\
	&=\cA e^{-i\eta \qop}(\Id-\epsilon \Pop - \epsilon\eta \qop).
\end{align}
So that we obtain  
\begin{align}
	\|\mathbf{A}_4\|_1\leq  \cA\|\Proj{N}^\perp e^{-i\eta \qop}(\Id-\epsilon \Pop - \epsilon \eta \qop)\rN\|_1.
\end{align}
As a consequence, we get $\|(\cD_{\LL_0}-\cD_{\Proj{N} \LL_0 \Proj{N}})(\rN)\|_1\leq f(\rN)$ with
\begin{align}
	\notag f(\rN)&=
	\xtr{\xP \rN \xP^\dagger}- \xtr{\Proj{N} \Uop \xP \rN \xP^\dagger \Uop^\dag \Proj{N}}+ 2\| \Proj{N}^\perp \Uop \xP \rN \xP^\dagger \Uop^\dag \Proj{N}\|_1\\
	\notag &+ \|(\xP^\dagger\xP - \Proj{N}\xP^\dagger \Uop_{N+1}^\dagger \Proj{N} \Uop_{N+1} \xP \Proj{N})\rN\|_1\\
	\notag &+\|\Proj{N}^\perp \Uop \xP \rN\|_1\\
	&+\cA\|\Proj{N}^\perp \Uop^\dag(\Id-\epsilon \Pop-\epsilon \eta \qop)\rN\|_1.
\end{align} 
To compute the errors due to the dissipators $\LL_k$, $1\leq k\leq 3$, note that as the unitary $\mathbf{R}$ commutes with $\create \destroy$, it commutes with its spectral projectors $\Proj{N},\Proj{N}^\perp$ and $\Proj{N+1}$. One can then check that
\begin{align}
	(\cD_{\LL_k}-\cD_{\Proj{N} \LL_k \Proj{N}})(\rN)=\mathbf{R}^k (\cD_{\LL_0}-\cD_{\Proj{N} \LL_0 \Proj{N}})(\mathbf{R}^{-k}\rN \mathbf{R}^{k})\mathbf{R}^{-k}.
\end{align}
Hence,
\begin{align}
	\|(\cD_{\LL_k}-\cD_{\Proj{N} \LL_k \Proj{N}})(\rN)\|\leq f(\mathbf{R}^{-k}\rN \mathbf{R}^k).
\end{align}
Eventually, we deduce that
\begin{align}
	\label{eq:estimate}
	\|\xL(\rN)-\xL_N(\rN)\|_1 \leq \sum_{k=0}^3 f(\mathbf{R}^{-k}\rN \mathbf{R}^k).
\end{align}

\paragraph{Numerical test}

\begin{figure}[hbt]
    \centering
    \includegraphics{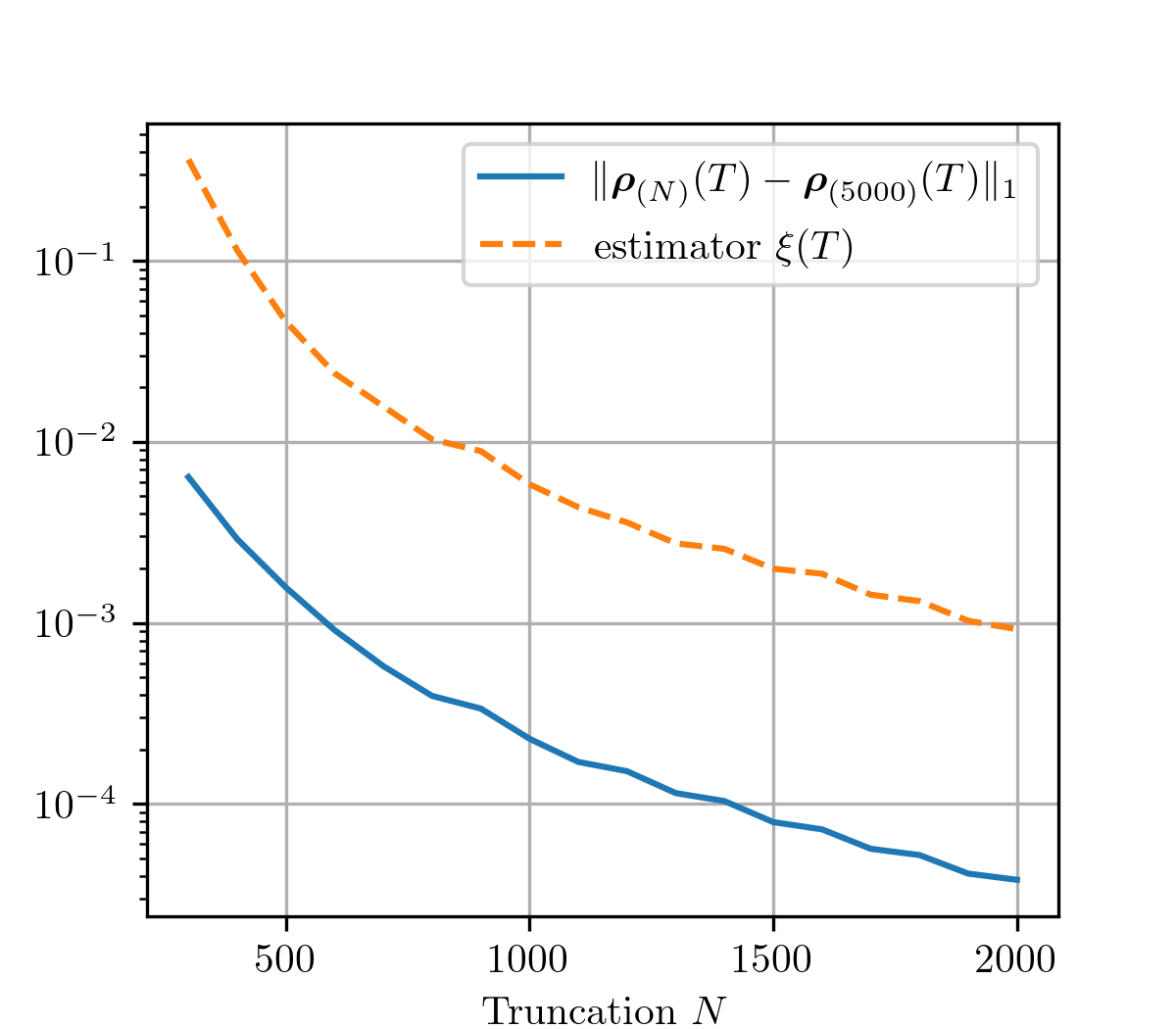}
    \captionsetup{justification=centering}
    \caption{$\|\boldsymbol{\rho}_{(5000)}(T) - \boldsymbol{\rho}_N(T)\|_1$ for various $N$ compared to the estimate $\xi(T)$ using \cref{eq:estimate}, for the simulations described at the end of \cref{subsec:GKP}.}
    \label{fig:estimator_GKP}
\end{figure}

We simulate the solution of \cref{eq:lindblad-GKP} with parameters $\epsilon=0.15$ and $\eta=2\sqrt{\pi}$ for various truncations in \cref{fig:estimator_GKP}, initiating the system in state $\boldsymbol{\rho}_0 = \ket{0}\bra{0}$. Time integration is performed using a second-order CPTP scheme, as described in \cite{robin2025unconditionallystabletimediscretization}, up to $T = 2\frac{1}{\epsilon \eta}$ with a fixed time step $\delta t = 5 \cdot 10^{-4} \cdot T$. We compare the results to the solution at $N=5000$ in \cref{fig:estimator_GKP}.

Due to the prohibitive computational cost of computing numerous trace-norms for such high dimensions, we were unable to perform the estimate for $N=5000$. Instead, we employ a 'naive' approach, comparing the results with those obtained for $N=4000$. This comparison yields a difference of $3 \cdot 10^{-6}$, providing a measure of confidence in our reference solution.

\subsection{Cosine operator}
\label{app-cos.tex}
We provide here an additional estimate related to unitary operators. In superconducting circuits, Josephson Junctions \cite{JOSEPHSON1962251} are commonly used non-linear elements. Their effect on the system Hamiltonian is typically to introduce a term of the form $\cos(\mathbf{O})$ (see e.g., \cite{devoretcQED}), where $\mathbf{O}$ is a self-adjoint operator, usually a sum of position and/or momentum operators acting on different modes. We assume that we have access to the truncation of the unitary operators $(e^{i\mathbf{O}})_N$, $(e^{2i\mathbf{O}})_N$, and their Hermitian conjugates, which is feasible for linear combinations of position and momentum operators. Denoting $\Uop=e^{i\mathbf{O}}$, we have $\cos(\mathbf{O}) = \frac{ \Uop + \Uop^\dag}{2}$.

Let us show that we can compute $\|(\cos(\mathbf{O}) - {(\cos(\mathbf{O}))}_N) \rN\|_1$. Using the decomposition $\xH=\xH_N\oplus \xH_N^\perp$, we introduce the notations: 
\begin{align}
    \Uop=
    \begin{pmatrix}
        \Uop_N&\Proj{N} \Uop \Proj{N}^\perp \\
        \Proj{N}^\perp \Uop \Proj{N}&\Proj{N}^\perp \Uop \Proj{N}^\perp 
    \end{pmatrix}\eqqcolon
    \begin{pmatrix}
        \mathbf{A}&\mathbf{C} \\
        \mathbf{B}&\mathbf{D}
    \end{pmatrix}.
\end{align}
We can now compute
\begin{align}
    \notag \|(\cos(\mathbf{O}) - {(\cos(\mathbf{O}))}_N) \rN\|_1 &= \frac12 \|((\Uop + \Uop^\dag) - (\Uop_N + \Uop_N^\dagger) )\rN\|_1\\
    \notag &= \frac12 \|(\Proj{N}^\perp \Uop + \Proj{N}^\perp \Uop^\dag)\rN\|_1\\
    \notag &= \frac12 \xtr{\sqrt{\rN(\mathbf{B}^\dagger + \mathbf{C})(\mathbf{B} + \mathbf{C}^\dagger)\rN}}\\
    &= \frac12 \xtr{\sqrt{\rN(\mathbf{B}^\dagger \mathbf{B} + \mathbf{C} \mathbf{C}^\dagger + \mathbf{B}^\dagger \mathbf{C}^\dagger + \mathbf{C}\mathbf{B})\rN}}.
\end{align}
As in \cref{th:unitary}, we have the following equalities that are obtained using $\Uop^\dag \Uop = \Uop \Uop^\dag = \Id$:
\begin{align}
        \mathbf{B}^\dagger \mathbf{B} &= \Id - \mathbf{A}^\dagger \mathbf{A},\\
        \mathbf{C} \mathbf{C}^\dagger &= \Id - \mathbf{A} \mathbf{A}^\dagger.
\end{align}
Hence, it remains to obtain $\mathbf{C}\mathbf{B}$ and its Hermitian conjugate. To this aim, we simply use that $\Proj{N}\Uop^2\Proj{N}=\mathbf{A}^2 + \mathbf{C}\mathbf{B}$. As a consequence, we obtain the following expression:
\begin{align}
    \|(\cos(\mathbf{O}) - {(\cos(\mathbf{O}))}_N) \rN\|_1 &= \frac12 \xtr{\sqrt{\rN(\Id - \Uop_N^\dagger \Uop_N + {(\Uop^2)}_N - (\Uop_N)^2 + h.c )\rN}}.
\end{align}

\section{Application 3: Space-adaptive solver for polynomial operators on bosonic modes}
\label{sec_faster}
\subsection{Dynamical reshapings, single mode}
\label{sec:reshapings}
Having an error estimate of the space truncation error not only allows us to bound the simulation's final error but also to monitor it throughout the time integration process. With this, we can detect when the truncated space is too small --causing significant error-- or overly large --resulting in wasted resources--. With this in mind, we developed an adaptive solver that dynamically adjusts the truncation size.
For the simulation on one bosonic mode, we propose the following algorithm which takes the following inputs:
\begin{enumerate}
\item $N_0\in \xN$, the initial truncation,
\item $\boldsymbol{\rho}_0\in \xK^1_+(\xH_{N_0})$ the initial state,
\item $T>0$ the final time, $\text{\texttt{space\_tol}}$ the space error tolerance, and $\text{\texttt{time\_tol}}$ the time solver tolerance (assumed to be small enough compared to $\text{\texttt{space\_tol}}$).
\item The functions $(t,N)\mapsto \Hamil_N(t)$ and $(t,N)\mapsto \LL^i_N(t)$ to construct $\xL_N$, as well as the required integer $w$ ensuring that $\xL(\rN)=\xL_{N+w}(\rN)$.
\item The parameters $w>1$ controlling the criteria for downsizing, and the decreasing and increasing size parameters $ n_-,n_+\in \xN$.
\end{enumerate}
We define $adaptive\_solve\_one\_step(N,\boldsymbol{\rho},t)$ to be a function that solves one discretisation step of the ODE $\frac{d}{dt}\boldsymbol{\rho}= \xL_N(\boldsymbol{\rho})$ with tolerance $\text{\texttt{time\_tol}}$. This function returns the chosen value of the time step and the value of the state at the following step.
We then perform the \cref{algo_1mode}.

\begin{algorithm}[h!]
\begin{algorithmic}[1]
    \State $N\gets N_0$
    \State $t\gets 0$
    \State $\xi \gets 0$
    \State $\boldsymbol{\rho} \gets \boldsymbol{\rho}_0$
    \While {$t \leq T$}
    \State $\delta \boldsymbol{\rho},\delta t \gets adaptive\_solve\_one\_step(N,\boldsymbol{\rho},t)$
    \State $\delta \xi \gets \| \xL(\boldsymbol{\rho}+\delta \boldsymbol{\rho})- \xL_N(\boldsymbol{\rho}+\delta \boldsymbol{\rho})\|_1$
        \If{$\xi+\delta \xi< (t+\delta t)/T* \text{\texttt{space\_tol}}$}
            \Comment{We accept the step}
            \State $\boldsymbol{\rho} \gets \boldsymbol{\rho} +\delta \boldsymbol{\rho}$
            \State $t \gets t +\delta t$
            \State $\xi \gets \xi +\delta \xi$
            \If{$\xi + \delta \xi + \|\boldsymbol{\rho}-\Proj{N-n_-}\boldsymbol{\rho} \Proj{N-n_-}\|_1 < (t+\delta t)/T* \text{\texttt{space\_tol}}/w$}\\
            \Comment{We downsize the state $\boldsymbol{\rho}$}
                \State $\boldsymbol{\rho} \gets \Proj{N-n_-} \boldsymbol{\rho} \Proj{N-n_-}$ \Comment{We reallocate $\boldsymbol{\rho}$ to a smaller matrix by deleting its tail}
                \State $ N \gets N-n_-$
                \State $\xi \gets \xi  + \|\Proj{N-n_-}\boldsymbol{\rho} \Proj{N-n_-}\|_1$
            \EndIf
        \Else \Comment{We rejected the step}
            \While{$\xi+\delta \xi< (t+\delta t)/T* \text{\texttt{space\_tol}}$}
                \State $\boldsymbol{\rho} \gets \Proj{N+n_+}\boldsymbol{\rho} \Proj{N+n_+}$ \Comment{We reallocate $\boldsymbol{\rho}$ on a bigger matrix by adding zeros}
                \State $N \gets N+n_+$
                \State $\delta \boldsymbol{\rho},\delta t \gets adaptive\_solve\_one\_step(N,\boldsymbol{\rho},t)$
                \State $\delta \xi \gets \| \xL(\boldsymbol{\rho}+\delta \boldsymbol{\rho})- \xL_N(\boldsymbol{\rho}+\delta \boldsymbol{\rho})\|_1$
            \EndWhile
            \Comment{We have accepted the step}
            \State $\boldsymbol{\rho} \gets \boldsymbol{\rho} +\delta \boldsymbol{\rho}$
            \State $t \gets t +\delta t$
            \State $\xi \gets \xi +\delta \xi$
        \EndIf

    \EndWhile
\end{algorithmic}
\caption{Space-adaptive solver}
\label{algo_1mode}
\end{algorithm}
    Observe that $\xi$ approximately solves the ODE
    \begin{align}
        \frac{d}{dt} \xi = \| \xL(\boldsymbol{\rho}) - \xL_N(\boldsymbol{\rho}) \|_1,& \ \xi(0)=0
    \end{align}
    which implies that, neglecting time-discretization error, it bounds the space-truncation error. The algorithm ensures that $\xi(t) \leq t \cdot \text{\texttt{space\_tol}} / T$, while adaptively trying to reduce the size of $\boldsymbol{\rho}$ to accelerate computation, albeit with decreased accuracy, when $\xi(t) \leq t \cdot \text{\texttt{space\_tol}} / (d \cdot T)$.
    
    Additionally, note that the time step is governed by the ODE of $\boldsymbol{\rho}$. To enhance the robustness of the interpolation, $\xi$ is updated implicitly using
    \begin{align}
        \xi(t + \delta t) - \xi(t) \approx \delta t \| \xL(\boldsymbol{\rho}(t + \delta t)) - \xL_N(\boldsymbol{\rho}(t + \delta t)) \|_1 .
    \end{align}

    \cref{fig:reshaping_overview} demonstrates the result over time of \cref{algo_1mode} when applied to the example described in \cref{ex:cat}. Initially, a large truncation is used (blue plots), which the algorithm subsequently reduces to enhance speed. Conversely, when starting with a low truncation (red plots), the algorithm increases it to improve accuracy.

    \begin{figure}[hbt]
        \centering
        \includegraphics[width=0.45\textwidth]{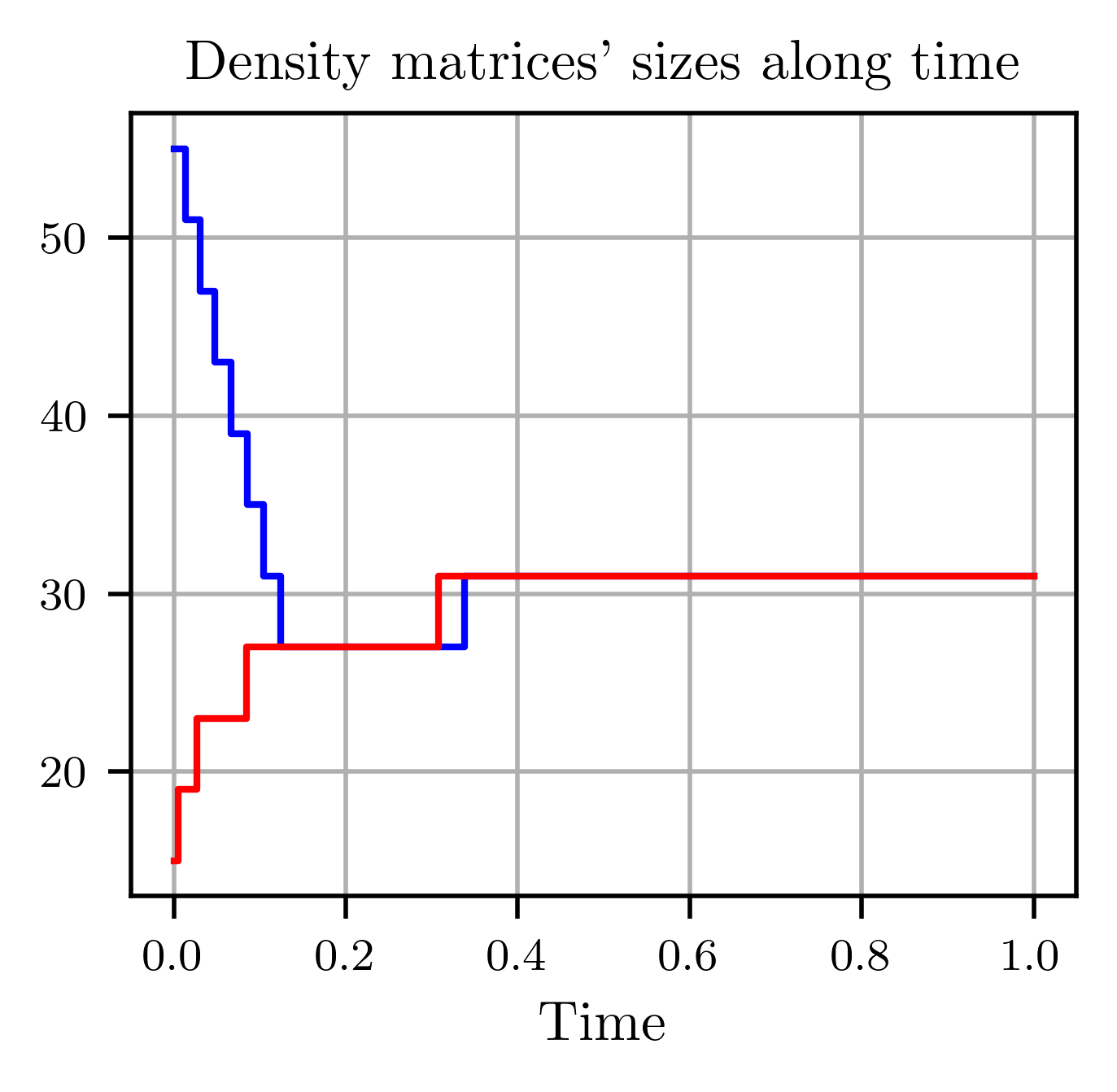}
        \hfill
        \includegraphics[width=0.45\textwidth]{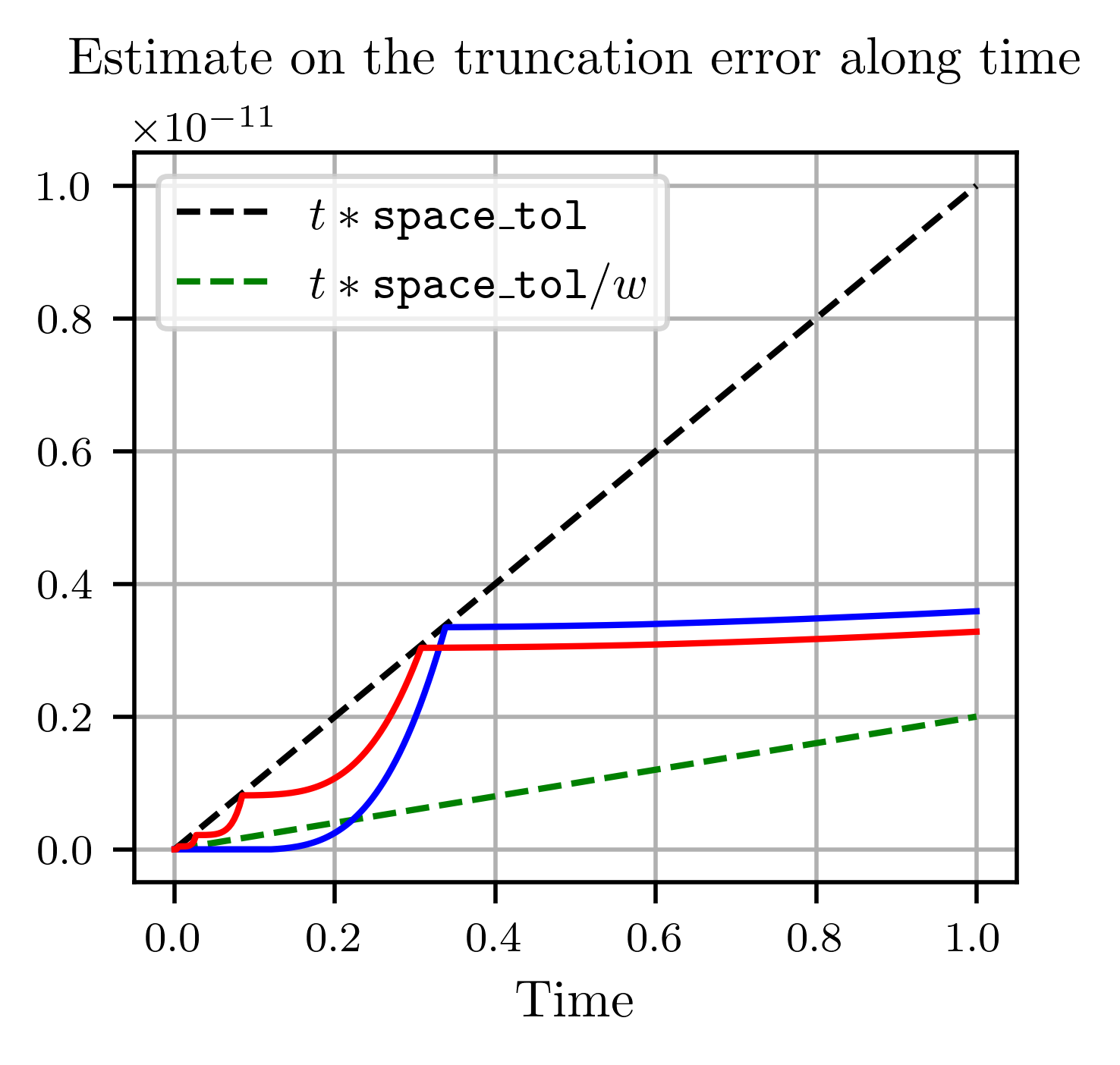}
    \caption{Left plot: Evolution of the sizes of the density matrices for two simulations, in red when starting from a small truncation $N=15$ and in blue when starting from a large truncation $N=55$. Right plot: The value of the estimators associated with the two simulations mentioned on the left figure together with the upper limit (black) and lower limit (green) for extending and reducing the matrix size. $w$ is 5 on this figure.
    Evolution of the size of $\rN(t)$ and the value of the estimator $\xi$ when following \cref{algo_1mode} on the example described in \cref{ex:cat} with two different initial truncations. Parameters of the simulations are $w=5$, $\text{\texttt{space\_tol}}=$1e-11, $n_+=4$, $n_-=4$, the time solver is an adaptive $\text{4}^{th}$ order Runge-Kutta with $\text{\texttt{time\_tol}}$ set at 1e-14.\\
    The red plots are associated to a simulation starting with a small truncation. In this case, the space estimates hit the upper bound for $\text{\texttt{space\_tol}}*t$ multiple times, leading to an enlargement of the truncation until the estimator is stabilized with a truncation size of 31. On the blue plots, associated with initial large truncation, we observe that the truncation size shrinks until its estimator hits the upper bound $\text{\texttt{space\_tol}}*t$, then it is enlarged one more time before stabilizing again to a size of 31. }
    \label{fig:reshaping_overview}
    \end{figure}

\subsection{Dynamical reshapings, several modes}
\label{sec:dynamical_severalmodes}
When dealing with $m$ bosonic modes, the previous algorithm can be straightforwardly implemented by replacing the scalar $N$ with the $m$-tuple $(N_1, \ldots, N_m)$. By fixing $n_\pm = (n_{1,\pm}, \ldots, n_{m,\pm})$, we obtain a direction to expand and reduce the size of the density matrix. However, the reduced Hilbert space $\xH_{(N_1, \ldots, N_m)} = \{ \otimes_{j=1}^m \ket{k_j} \mid 0 \leq k_j \leq N_j \}$ is not always a good choice for simulations. For example, instead of bounding the number of excitations of each mode separately, we can bound the total number of excitations, leading to the following definition: $\xH_{N_{\textrm{tot}}} = \{ \otimes_{j=1}^m \ket{k_j} \mid 0 \leq \sum_{j=1}^m k_j \leq N_{\textrm{tot}} \}$.

These choices are problem-dependent. For example, in \cref{ex:catplusbuffer}, the Hamiltonian includes a two-photon exchange term $(\create)^2 \destroyb + \destroy^2 \createb$. As this term commutes with the operator $\frac{\create \destroy}{2}+ \createb \destroyb$, it leads to the natural choice of truncation $\xH_{N_{\textrm{cat}}} = \{ \ket{k_1} \otimes \ket{k_2} \mid 0 \leq k_1/2 + k_2 \leq N_{\textrm{cat}} \}$. Note that in \cref{fig:estimator_2D}, where the truncated Hilbert space is $\xH_{(N_1,N_2)}$, the truncation error is mostly determined by biggest $N_\textrm{cat}$ such that $\xH_{N_\textrm{cat}}\subset \xH_{N_1,N_2}$. \cref{fig:reshaping_2D_overview} illustrates the result of the adaptation of \cref{algo_1mode} with the single truncation parameter $N_{\textrm{cat}}$.
\begin{figure}[h!]
    \centering
    \includegraphics[width=0.45\textwidth]{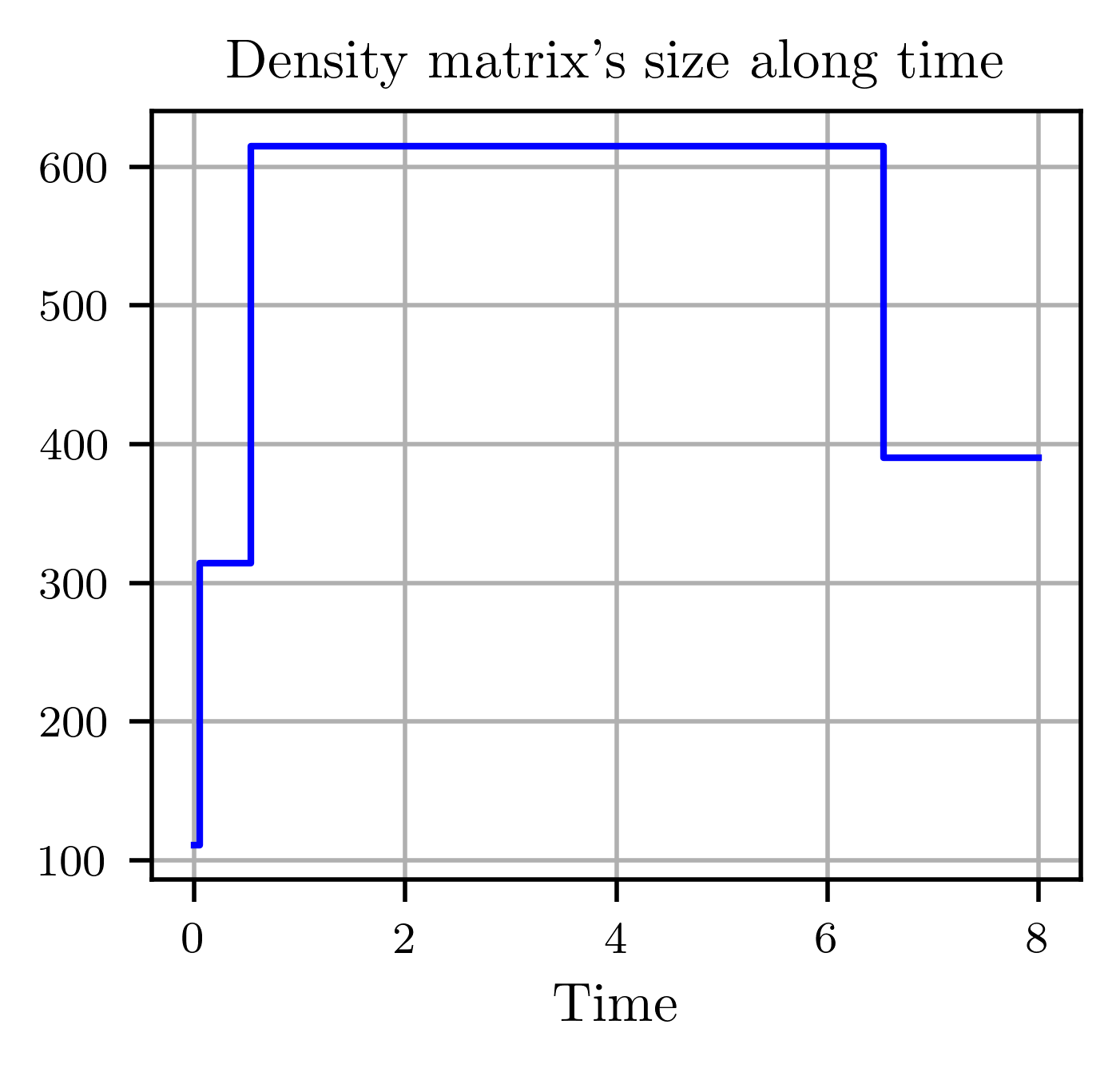}
    \hfill
    \includegraphics[width=0.45\textwidth]{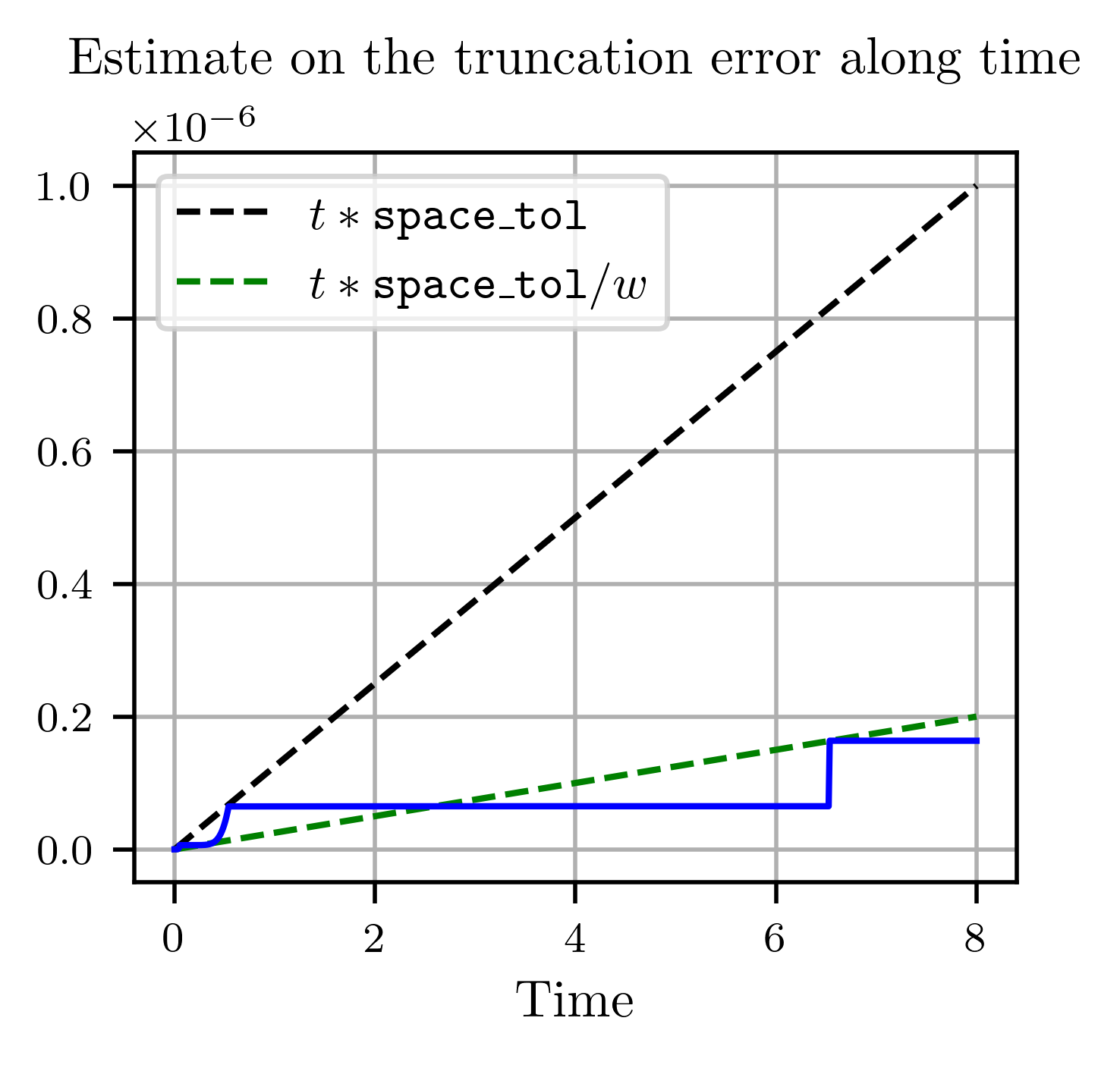}
    \caption{Left plot: Evolution of the size of the Hilbert space. Right plot: The value of the estimator associated with the simulation described below together with the upper limit (black) and lower limit (green) for extending and reducing the matrix size. Evolution of the size of $\rN(t)$ and the value of the estimator $\xi$ when following \cref{algo_1mode} on a slightly different system than the example described in \cref{ex:catplusbuffer}: The parameter $\alpha$ is set at $1.5$ for $t$ in $(0,1.5)$, then $\alpha=0$ until the end. The parameters of the simulations are $w=5$, $\text{\texttt{space\_tol}}=$1e-11, $n_+=7$, $n_-=5$, the time solver is an adaptive $\text{4}^{th}$ order Runge-Kutta with $\text{\texttt{time\_tol}}$ set at 1e-14.\\
    As the initial state is $\rho_0=\ket{0}\bra{0}\otimes \ket{0}\bra{0}$, setting $\alpha=1.5$ increases the population of the excited Fock states, meaning that the required truncation should increase. In a second time, the state should converge toward the vacuum, meaning that it is interesting to decrease the size of the truncated Hilbert space. In the simulation, we start with $\xH_{N_{\textrm{cat}}} = \{ \ket{k_1} \otimes \ket{k_2} \mid 0 \leq k_1/2 + k_2 \leq N_{\textrm{cat}} \}$ with $N_{\textrm{cat}}=6$, we observe that the truncation is first increasing to capture the true solution up to the estimate's tolerance (here 1e-6). When $\alpha$ is set to 0, no error is made due to the dynamics. Around $t=6.75$, the density matrix is truncated, leading to a loss of the truncated information. We have to add $\|\rN(t)-\Proj{N-n_-} \rN(t) \Proj{N-n_-}\|_1$ to the estimate, this is the jump one can see.
    }
    \label{fig:reshaping_2D_overview}
\end{figure}

\subsection{\textsc{Dynamiqs\_adaptive}, a new library to perform space adaptive simulations}
\label{sub_sec:dynamiqs_adaptive}
The first author developed an extension of the library \textsc{Dynamiqs} \cite{guilmin2024dynamiqs}, which is a python library for high-performance quantum systems simulation, using JAX and GPU acceleration. This new library is named \textsc{dynamiqs\_adaptive} and implements the following features:
\begin{itemize}
    \item Running a bosonic quantum simulation with any Runge-Kutta (possibly time-adaptive) solvers, with a theoretical warranty on the truncation errors made by the simulation, using the estimator developed above \cref{eq_estimate_space} and showcased in \cref{sec_ex1} and \cref{sec_faster}.
    \item Space-adaptive simulations using \cref{algo_1mode} for one or several bosonic modes with polynomial operators. Several options are available for the space truncations on multi-modes simulation.
\end{itemize}
	\textsc{dynamiqs\_adaptive} is available on github at the following address \url{https://github.com/etienney/dynamiqs_adaptive/} and is distributed under the open-source license Apache 2.0.

 \section{Time and space-time estimates}
\label{sec:time_estimator}

Up to this point, we have focused solely on the space-truncation error. In this part, we do not assume that we have access to $(\rN(t))_{t\geq0}$ but only to the discrete time trajectory $(\boldsymbol{\rho}_{n,\delta t})_{n\geq 0}$. Note that the space estimate we derived in \cref{prop-aposteriori} cannot be easily combined with a time estimate. Indeed, if $\xL$ is an unbounded operator, it is not continuous from $\xK^1$ to $\xK^1$. This implies that $\|\xL(\tilde{ \boldsymbol{\rho}}) - \xL(\boldsymbol{\rho})\|_1$ might be large even for arbitrarily small $\|\tilde{ \boldsymbol{\rho}} - \boldsymbol{\rho}\|_1$, preventing a naive approximation of the integral in \cref{prop-aposteriori} by a $\xK^1$-approximation of $(\rN(t))_{t\geq0}$. We address this issue by presenting space-time estimates. We start by proving that the total distance between the continuous solution $e^{T\xL}\boldsymbol{\rho}_0$ and the space and time discretized approximation $\xF_{\delta t}^{N_{step}}(\boldsymbol{\rho}_0)$ is smaller than the sum of errors made in each time-step: assuming that $N_{step}\delta t =T$, we have the following inequalities that start with a telescoping argument:
\begin{align}
    \notag \|e^{T\xL}\boldsymbol{\rho}_0- \xF^{N_{step}}_{\delta t}(\boldsymbol{\rho}_0)\|_1 &= \left\|\sum_{n=0}^{N_{step}-1} e^{(N_{step}-n)\delta t \xL} \xF^{n}_{\delta t}\boldsymbol{\rho}_0- e^{(N_{step}-n-1)\delta t \xL} \xF^{n+1}_{\delta t}\boldsymbol{\rho}_0 \right\|_1\\
    \notag &\leq \sum_{n=0}^{N_{step}-1} \left\| e^{(N_{step}-n-1)\delta t \xL} (e^{\delta t \xL} \xF^{n}_{\delta t}\boldsymbol{\rho}_0- \xF^{n+1}_{\delta t}\boldsymbol{\rho}_0) \right\|_1\\
    &\leq \sum_{n=0}^{N_{step}-1} \left\|e^{\delta t \xL} \xF^{n}_{\delta t}\boldsymbol{\rho}_0- \xF^{n+1}_{\delta t}\boldsymbol{\rho}_0 \right\|_1,
\end{align}
where we used that the CPTP map $e^{\delta t \xL}$ contracts the trace-norm.
If we denote again $\boldsymbol{\rho}_{n,\delta t}= \xF_{\delta t}^{n}\boldsymbol{\rho}_0$, we got
\begin{align}
    \label{eq_sum_error_dt}
    \|e^{N_{step}\delta t \xL}\boldsymbol{\rho}_0- \xF^{N_{step}}_{\delta t}(\boldsymbol{\rho}_0)\|_1 \leq \sum_{n=0}^{N_{step}-1} \left\|(e^{\delta t \xL}- \xF_{\delta t})\boldsymbol{\rho}_{n,\delta t}\right\|_1.
\end{align}
To compute a bound on the terms inside the sum, we restrict ourselves to the following cases:
\begin{itemize}
    \item Time-independent Lindbladian with a time solver based on $k^{th}$ order Taylor truncation.
    \item Time-dependent Lindbladian with a first order explicit Euler scheme.
\end{itemize}

\subsection{\texorpdfstring{Time-independent Lindbladian with $k^{th}$ order Taylor scheme}{Time-independent Lindbladian with kth order Taylor scheme}}

\label{paragraph:time_invariant_taylor}
For a fixed integer $k \geq 1$, we define the following discrete scheme that is simply the truncation of the Taylor expansion of $e^{\delta t \xL_N}$, which we call the $k^{\text{th}}$ order Taylor scheme:

\begin{align}
    \xF_{\delta t}(\boldsymbol{\rho})= \sum_{j=0}^{k} \frac{\delta t^j}{j!} \xL_N^j(\boldsymbol{\rho}).
\end{align}
Then, we have
\begin{lemma}
    \label{lemma_taylor}
    For any $\boldsymbol{\rho} \in \xK_s^1(\xH_N)$ that belongs to the domain of $\xL^{k+1}$, we have
    \begin{align}
        \|\xF_{\delta t}(\boldsymbol{\rho})-e^{\delta t \xL}(\boldsymbol{\rho})\|_1 \leq \left\|\xF_{\delta t}(\boldsymbol{\rho})- \sum_{j=0}^{k} \frac{\delta t^j}{j!}\xL^j(\boldsymbol{\rho})\right\|_1 + \frac{\delta t^{k+1}}{(k+1)!} \left\|\xL^{k+1}(\boldsymbol{\rho})\right\|_1.
    \end{align}
\end{lemma}
Before proving this Lemma, note that together with \cref{eq_sum_error_dt}, it implies that
\begin{align}
    \|e^{T\xL}\boldsymbol{\rho}_0- \xF^{N_{step}}_{\delta t}(\boldsymbol{\rho}_0)\|_1 \leq \sum_{n=0}^{N_{step}-1} \left( \left\|\sum_{j=0}^{k} \frac{\delta t^j \xL^j(\boldsymbol{\rho}_{n,\delta t})}{j!}-\xF_{\delta t}(\boldsymbol{\rho}_{n,\delta t})\right\|_1 + \frac{\delta t^{k+1}}{(k+1)!} \left\|\xL^{k+1}(\boldsymbol{\rho}_{n,\delta t})\right\|_1 \right) .
\end{align}

\begin{remark}
    In the case where $\xH$ is finite-dimensional and $\xH_N = \xH$, we have $\| \sum_{j=0}^{k} \frac{\delta t^j}{j!} \xL^j(\boldsymbol{\rho}) - \xF_{\delta t}(\boldsymbol{\rho}) \|_1 = 0$. Consequently, the estimate contains only the terms $\frac{\delta t^{k+1}}{(k+1)!} \|\xL^{k+1}(\boldsymbol{\rho}_{n,\delta t})\|_1$.
\end{remark}
\begin{proof}[Proof of \cref{lemma_taylor}]
    With the triangle inequality, we get
    \begin{align}\|\xF_{\delta t}(\boldsymbol{\rho})-e^{\delta t \xL}(\boldsymbol{\rho})\|_1 \leq \left\|\xF_{\delta t}(\boldsymbol{\rho})-\sum_{j=0}^{k} \frac{\delta t^j}{j!}\xL^j(\boldsymbol{\rho})\right\|_1 +  \left\|\sum_{j=0}^{k} \frac{\delta t^j}{j!}\xL^j(\boldsymbol{\rho}) - e^{\delta t \xL}(\boldsymbol{\rho})\right\|_1.
    \end{align}
    To deal with the second term, we use the integral form of the remainder of the Taylor expansion
    \begin{align}
        \notag \left\|e^{\delta t \xL}(\boldsymbol{\rho})- \sum_{j=0}^{k} \frac{\delta t^j}{j!}\xL^j(\boldsymbol{\rho})\right\|_1& = \left\|\frac{1}{k!}\int_0^{\delta t} \xL^{k+1}(e^{s\xL}\boldsymbol{\rho})(\delta t-s)^k ds\right\|_1\\
        \notag &= \left\|\frac{1}{k!}\int_0^{\delta t} e^{s\xL}\xL^{k+1}(\boldsymbol{\rho})(\delta t-s)^k ds\right\|_1\\
        \notag &\leq \frac{1}{k!}\int_0^{\delta t}\| \xL^{k+1}(\boldsymbol{\rho})(\delta t-s)^k ds\|_1\\
        &=\frac{\delta t^{k+1}\|\xL^{k+1}(\boldsymbol{\rho})\|_1}{(k+1)!}.
    \end{align}
\end{proof}

\subsection{Time-dependent Lindbladian with a first order explicit Euler scheme.}
\label{subsubsec_time_dep}
Let us now consider a time-dependent Lindbladian that we denote $\xL(t,\cdot)$. For $t_1 \leq t_2$, the flow of the Lindbladian is the bounded linear map $\Phi(t_1,t_2,\cdot): \xK^1\to \xK^1$ characterized by the following set of equations for all $\boldsymbol{\rho}_0$ smooth enough\footnote{See e.g. \cite[Section 3.2]{gondolfEnergyPreservingEvolutions2023}}:
\begin{align}
    \label{eq_lind_time_dependant}
    \partial_{t_2}\Phi(t_1,t_2,\boldsymbol{\rho}_0)=\xL(t_2,\Phi(t_1,t_2,\boldsymbol{\rho}_0)), \quad \Phi(t_1,t_1,\boldsymbol{\rho}_0)=\boldsymbol{\rho}_0, \quad \forall t_2 \geq t_1
\end{align}

The first order explicit Euler scheme is
\begin{align}
    \xF_{\delta t}(t,\boldsymbol{\rho})=\boldsymbol{\rho} + \delta t \xL_N(t,\boldsymbol{\rho}).
\end{align}
The applications $\boldsymbol{\rho} \mapsto \Phi(t_1,t_2,\boldsymbol{\rho})$ are CPTP map, so \cref{eq_sum_error_dt} is modified as follows:
\begin{align}
    \|\boldsymbol{\Phi}(0,T,\boldsymbol{\rho}_0)- \boldsymbol{\rho}_{N_{step},\delta t}\|_1 \leq \sum_{n=0}^{N_{step}-1} \|\boldsymbol{\Phi}(t_n,t_{n+1},\boldsymbol{\rho}_{n,\delta t})- \xF_{\delta t}(t_n,\boldsymbol{\rho}_{n,\delta t})\|_1,
\end{align}
with $t_n=n\delta t$ and
\begin{align}
    \boldsymbol{\rho}_{n+1,\delta t}= \xF_{\delta t}(n\delta t,\boldsymbol{\rho}_{n,\delta t}),\quad \boldsymbol{\rho}_{0,\delta t}=\boldsymbol{\rho}_0.
\end{align}

\begin{lemma}
    \label{lem_time_dep_estimate}
    The norm of the error on each time step is bounded by
    \begin{align}
        \label{eq_lem_time_dep_estimate}
        \notag
        \|\boldsymbol{\Phi}(t_n,t_{n+1},\boldsymbol{\rho}_{n,\delta t})- \xF_{\delta t}(t_n,\boldsymbol{\rho}_{n,\delta t})\|_1 \leq \delta t \sup_{t_n \leq s \leq t_{n+1}} \| (\xL(s,\cdot)-\xL(t_n,\cdot))\boldsymbol{\rho}_{n,\delta t}\|_1 \\
        + \frac{\delta t ^2}{2} \sup_{t_n \leq s \leq t_{n+1}} \|\xL(s,(\xL(t_n,\boldsymbol{\rho}_{n,\delta t})))\|_1+ \delta t \|(\xL(t_n,\cdot)-\xL_N(t_n,\cdot))\boldsymbol{\rho}_{n,\delta t}\|_1.
    \end{align}
\end{lemma}

\begin{proof}
    To lighten the notation, we focus on the first time-step and denote $\boldsymbol{\rho}_t=\Phi(t,0,\boldsymbol{\rho}_0)$. We introduce the error between the exact solution and the first order euler scheme using $\xL$ and not $\xL_N$:
    \begin{align}
        \mathbf{r}(t)=\boldsymbol{\rho}_{t}-\boldsymbol{\rho}_{0} -t \xL(0,\boldsymbol{\rho}_0).
    \end{align}
    Taking the time derivative, we have,
    \begin{align}
        \notag \dot{\mathbf{r}}(t)&=\xL(t,\boldsymbol{\rho}(t))-\xL(0,\boldsymbol{\rho}_0)\\
        &=\xL(t,\mathbf{r}(t) ) + \xL(t, \boldsymbol{\rho}_0+ t \xL(0,\boldsymbol{\rho}_0))-\xL(0,\boldsymbol{\rho}_0).
    \end{align}
    Using Duhamel principle, we get
    \begin{align}
        \mathbf{r}(t)=\boldsymbol{\Phi}(0,t,\mathbf{r}(0))+\int_0^t\boldsymbol{\Phi}(s,t,\xL(s, \boldsymbol{\rho}_0+ s \xL(0,\boldsymbol{\rho}_0))-\xL(0,\boldsymbol{\rho}_0)) ds.
    \end{align}
    As $\Phi(s,t,\cdot)$ contracts the trace norm and $\mathbf{r}(0)=0$, we get
    \begin{align}
        \notag \|\mathbf{r}(t)\|_1 &\leq \int_0^t \|\xL(s, \boldsymbol{\rho}_0 + s \xL(0,\boldsymbol{\rho}_0))-\xL(0,\boldsymbol{\rho}_0)\|_1\\
        &\leq \int_0^t \|(\xL(s,\cdot)-\xL(0,\cdot))(\boldsymbol{\rho}_0)\|_1 + \|s \xL(s,\xL(0,\boldsymbol{\rho}_0))\|_1 ds.
    \end{align}
    Hence,
    \begin{align}
        \|e(\delta t)\|_1 &\leq\delta t \sup_{0 \leq s \leq \delta t} \| (\xL(s,\cdot)-\xL(0,\cdot))\boldsymbol{\rho}_0\|_1 + \frac{\delta t ^2}{2} \sup_{0 \leq s \leq \delta t} \|\xL(s,(\xL(0,\boldsymbol{\rho}_0)))\|_1.
    \end{align}
    As 
    \begin{align}
        \notag \boldsymbol{\Phi}(0,\delta t,\boldsymbol{\rho}_0)- \xF_{\delta t}(0,\boldsymbol{\rho}_0)&=\mathbf{r}(\delta t)+\boldsymbol{\rho}_0 + \delta t \xL(0,\boldsymbol{\rho}_0)- \boldsymbol{\rho}_0 - \delta t \xL_N(0,\boldsymbol{\rho}_0)\\
        &=\mathbf{r}(\delta t) + \delta t(\xL(0,\boldsymbol{\rho}_0)-\xL_N(0,\boldsymbol{\rho}_0)),
    \end{align}
    we use a triangle inequality to finish the proof of \cref{lem_time_dep_estimate}.
\end{proof}
Note that in the time-independent case, we recover 
\begin{align}
\|\mathbf{r}(\delta t)\|_1 \leq \frac{\delta t^2}{2} \|\xL(0,\xL(0,\boldsymbol{\rho}_0))\|_1+\delta t \|(\xL(0,\cdot)-\xL_N(0,\cdot))\boldsymbol{\rho}_0\|_1.
\end{align}
corresponding to \cref{lemma_taylor} for $k=1$.

It is important to notice that this estimate requires prior knowledge on the regularity of $s\mapsto \xL(s,\cdot)$. A typical case is when the Lindbladian (or part of it) is of the form $u(t)\xL_0$ where $\xL_0$ is time-independent. For example, if $u$ is a $\mathcal{C}^1$ function, then
\begin{align}
    \sup_{0 \leq s \leq \delta t} \| (\xL(s,\cdot)-\xL(0,\cdot))\boldsymbol{\rho}_0\|_1 &\leq \delta t \|u'\|_{L^\infty([0,\delta t])} \|\xL_0(\boldsymbol{\rho}_0)\|_1,\\
    \sup_{0 \leq s \leq \delta t} \|\xL(s,(\xL(0,\boldsymbol{\rho}_0)))\|_1 & \leq |u(0)| \|u\|_{L^\infty([0,\delta t])} \|\xL_0^2(\boldsymbol{\rho}_0)\|_1.
\end{align}
\subsection{Applications}
\subsubsection{\texorpdfstring{Frequency driven damped oscillator revisited - $\Hamil = u(t) \create \destroy$ and $\cD_{\destroy}$}{Frequency driven damped oscillator revisited}}
\label{ex:aadag_st}
Let us now generalize the truncation estimate obtained in \cref{ex:aadag}. We assume that we use a first order explicit Euler scheme for the time integration, and that $u(t)$ is a $\mathcal{C}^1$ function.
Then, following \cref{subsubsec_time_dep}, we can bound the error between the continuous solution $(\boldsymbol{\rho}_t)_{t\geq 0}$ and the discrete time approximation with step-size $\delta t$, $(\boldsymbol{\rho}_{n,\delta t})_{n\in \xN}$, by
\begin{align}
    \notag
    \|\boldsymbol{\rho}_{n\delta t}-\boldsymbol{\rho}_{n,\delta t}\|_1 &\leq \delta t ^2 \|u'\|_\infty \sum_{k=0}^n \|[\create_N \destroy_N,\boldsymbol{\rho}_{n,\delta t}]\|_1 \\
    &+\frac{\delta t^2}{2} \sum_{k=0}^n \|\cD[\destroy_N](\xL_N(k \delta t,\boldsymbol{\rho}_{k,\delta t}))\|_1 + \|u\|_\infty \|[\create_N \destroy_N,(\xL_N(k \delta t,\boldsymbol{\rho}_{k,\delta t}))]\|_1.
\end{align}

\subsubsection{\texorpdfstring{Driven oscillator revisited - $\Hamil = u(t)(\destroy + \create)$}{Driven oscillator revisited}}
\label{ex:aplusadag_st}
In \cref{ex:aplusadag}, we had the space estimate
\begin{align}
    \|\boldsymbol{\rho}_t-\rN(t)\|_1 \leq \int_0^t 2|u(s)|\sqrt{N+1} \sqrt{\bra{N}{\rN(s)}^2 \ket{N}}ds.
\end{align}
To obtain a space-time estimate for a first order explicit Euler scheme, we get from \cref{lem_time_dep_estimate}
\begin{align}
    \notag \|\boldsymbol{\rho}_{n\delta t}-\boldsymbol{\rho}_{n,\delta t}\|_1 &\leq \delta t ^2 \|u'\|_\infty \sum_{k=0}^n \|[\destroy_{N+1}+\create_{N+1},\boldsymbol{\rho}_{k,\delta t}]\|_1\\
    \notag &+ \delta t ^2 \|u\|_\infty \sum_{k=0}^n |u(k\delta t)|\|[\destroy_{N+2}+\create_{N+2},[\destroy_{N+1}+\create_{N+1}, \boldsymbol{\rho}_{k,\delta t}]]\|_1\\
    &+2\delta t \sum_{k=0}^n |u(k\delta t)|\sqrt{N+1} \sqrt{\bra{N}{\boldsymbol{\rho}_{k,\delta t}}^2 \ket{N}}.
\end{align}
Note that it is needed to perform computation in the space $\xK^1_s(\xH_{N+2})$ for the second sum.

\subsubsection{General bosonic modes}
In this short subsection, we explain how to extend the result of \cref{sec_ex1}, and briefly show that for a bosonic mode with polynomials in creation and annihilation, we can also compute the space-time estimate. We only focus on the one mode case, the generalization to several modes is similar to \cref{subsec:ap1general_case}.
As a consequence of \cref{co:Hamiltonian,co:dissipator}, we get:
\begin{corollary}
    \label{co:L-LN}
    If the Hamiltonian $\Hamil$ and the dissipators $\LL_i$, are polynomials in $\destroy$ and $\create$, of degree $d_{\Hamil}$ and $d_i$ resp., then defining $d=\max(d_H, 2\max_i(d_i))$ we have:
    \begin{align}
        (\mathcal{L}-\mathcal{L}_N)(\boldsymbol{\rho}_N)&=(\mathcal{L}_{N + d}-\mathcal{L}_N)(\boldsymbol{\rho}_N)\in \xK_s^1(\xH_{N+d}).
    \end{align} 
 This naturally extends to 
    \begin{align}
        (\mathcal{L}^k-\mathcal{L}_N^k)(\boldsymbol{\rho}_N)&=(\mathcal{L}_{N + kd}^k-\mathcal{L}_N^k)(\boldsymbol{\rho}_N) \in \xK_s^1(\xH_{N+kd}).
    \end{align} 
\end{corollary}
Hence, we can always numerically compute the estimate of \cref{lemma_taylor} for an order $k$ scheme by embedding $\boldsymbol{\rho}_{n,\delta t}$ in $\xK_s^1(\xH_{N+(k+1)d})$, and computing the trace norm of $(\mathcal{L}^j_{N + (k+1)d}-\mathcal{L}^j_N)(\boldsymbol{\rho}_{n,\delta t})$ for $1\leq j \leq k$ and of $\mathcal{L}^{k+1}_{N + (k+1)d}\boldsymbol{\rho}_{n,\delta t}$ in this finite dimensional space.

\section{Conclusion}
\label{sec_ccl}
In this article, we provided \textit{a posteriori} computable error bounds for the space-truncation error and/or time discretization. We have demonstrated numerically the efficiency of our approach for a large class of problems involving bosonic modes.

Several promising avenues for future research are worth pursuing:

\begin{enumerate}
\item In \cref{sec_ex2}, we have demonstrated how space estimates can be applied to various operators related to unitary operators. Extending these methods to other types of systems with different dissipators would be a valuable next step.
\item We have not addressed the proof of convergence of the approximations toward the solution $(\boldsymbol{\rho}_t)_{t\geq 0}$ in this work. This problem has been initiated for time discretization with unbounded operators in \cite{robin2025unconditionallystabletimediscretization} and in \cite{robin2025convergenceanalysisgalerkinapproximations} for space discretization, but space-time discretization remains an open question.
\item A powerful property that would be beneficial to prove with \textit{a posteriori} error bounds is an efficiency result. Specifically, while we know that the error is smaller than the estimate, it would be advantageous to prove the existence of a constant such that the estimate is controlled by this constant times the error. Although such a property may not hold universally, identifying verifiable criteria that ensure its validity would be a significant advancement.
\item We believe that maintaining and extending the library \textsc{dynamiqs\_adaptive} (see \cref{sub_sec:dynamiqs_adaptive}) should be a valuable asset for the community. Several improvements, both on the numerical and algorithmic side could be implemented. In particular, general support of the space-time estimates would be very interesting.
\item In this article, we have restricted ourselves to linear approximations for space discretization, specifically using a Hilbert space $\xH_N$. Many recent works, however, consider non-linear approximations, such as low-rank approximations, tensor networks, and many others. These settings are fundamentally different as they do not yield a Lindblad equation on the approximation manifold. Nevertheless, investigating \textit{a posteriori} error estimates in these contexts would be extremely interesting and could provide valuable insights.
\item Another worth considering extension is the stochastic unraveling of the Lindblad master equation. In this case, providing either strong or weak \textit{a posteriori} estimates on the truncation and/or time discretization errors has not been explored yet.
\end{enumerate}

\paragraph{Acknowledgements}

The authors would like to express their gratitude to Alexandre Ern and Claude Le Bris for their invaluable and fruitful discussions. Additionally, we extend our thanks to Ronan Gautier, Pierre Guilmin, Mazyar Mirrahimi, Alexandru Petrescu, Alain Sarlette, Lev-Arcady Sellem, and Antoine Tilloy for their insightful feedback.

This project has received
funding from the European Research Council (ERC) under the European Union’s Horizon
2020 research and innovation program (grant agreement No. 884762) and Plan France 2030 through the project ANR-22-PETQ-0006.
\appendix
\crefalias{section}{appendix}
\section{Some pathological examples}
\label{sec:pathological}
A classical approach to get a rough idea of the truncation error is to compare numerical solutions while increasing the truncation dimension $N$, i.e. comparing $\boldsymbol{\rho}_N$ and $\boldsymbol{\rho}_{N+k}$. While we do not claim that in most of the reasonable cases this naive approach is misleading, we provide here some pathological examples where it fails.

First, due to symmetries, it is usually a bad idea to compare two simulations with truncation dimensions $N$ and $N+k$ with $k$ too small. For example, if the initial condition is even in the Fock basis, and if the Lindbladian preserves this parity, then all simulations with $2N$ and $2N+1$ will give the exact same result; meaning that for all $N$, $\boldsymbol{\rho}_{(2N)}=\boldsymbol{\rho}_{(2N+1)}$ independently of the convergence of $\rN$ towards the exact solution. More generally, as soon as $\xL_N$ and $\xL_{N+k}$ coincide on $\xK(\xH_N)$, the two simulations will give the same result. Using for example only monomials of the form $\boldsymbol{a}^{k+1}, \boldsymbol{a}^{\dagger, k+1}$, in the Lindbladian will lead to this issue.

A more striking example appears in~\cite{ashhab2025finitedimensionalapproximationsgeneralizedsqueezing}. There the authors study the Hamiltonian
\begin{equation*}
    \Hamil = i \left({\boldsymbol{a}^\dagger}^3 - \boldsymbol{a}^3\right),
\end{equation*}
which is not essentially self-adjoint on the finite span of Fock states. Because the dynamics commutes with the rotation by $e^{2\pi i/3}$, the Hilbert space splits into three invariant subspaces. Starting from the vacuum, it is therefore natural to truncate on

\begin{equation}\label{eq:N_Burgath}
    \xH_{3N}=\operatorname{Span}\{\ket{0},\ket{3},\ldots,\ket{3N}\}, \quad N\in\mathbb{N}.
\end{equation} 

Ashhab et al. show that the unitary evolutions generated by the truncated Hamiltonians on $\xH_{3 (2N)}$ and on $\xH_{3 (2N+1)}$ each converge as $N\to\infty$, but to two different limits -- namely, the unitary evolutions associated with two distinct self-adjoint extensions of $H$. Hence, comparing $\psi_{3N_1}$ and $\psi_{3N_2}$ misleadingly suggests convergence of the simulations if the integers $N_1$ and $N_2$ have the same parity. In contrast, our estimator detects a nontrivial interaction with the tail of the Hilbert space and therefore signals that the naive comparison is unreliable, see \cref{fig:fullwidth_example}.

More precisely, our estimator trivially applies to the Schr\"odinger equation, and for any choice of self-adjoint extension of $H$, we have
\begin{align}
    \label{eq:pathological_estimator}
    \|\ket{\psi(t)}-\ket{\psi(t)}_N\| &\leq \|\ket{\psi(0)}-\ket{\psi(0)}_N\| + \int_0^t \|\left(\Hamil-\Hamil_N \right) \ket{\psi(s)}_N\| ds.
\end{align}

Note that while $\psi(t)$ depends on the choice of self-adjoint extension, the estimator does not, since it only involves applying $\Hamil_N$ and $\Hamil$ to $\ket{\psi(s)}_N$, which belongs to the finite span of Fock states. This means that it does not depend on the choice of the self-adjoint extension of $\Hamil$.

As a consequence, we can define the estimator $\xi_S$ as
\begin{align*}
    \frac{d\xi_S(t)}{dt} = \|\left(\Hamil-\Hamil_N\right) \ket{\psi(t)}_N\|, \quad \xi_S(0)  = \|\ket{\psi(0)}-\ket{\psi(0)}_N\|,
\end{align*}
    and we get $\|\ket{\psi(t)}-\ket{\psi(t)}_N\| \leq \xi_S(t)$ for all $t\geq 0$, and for any choice of self-adjoint extension of $\Hamil$.

\begin{figure}[ht]
    \centering
    \includegraphics[width=0.9\textwidth]{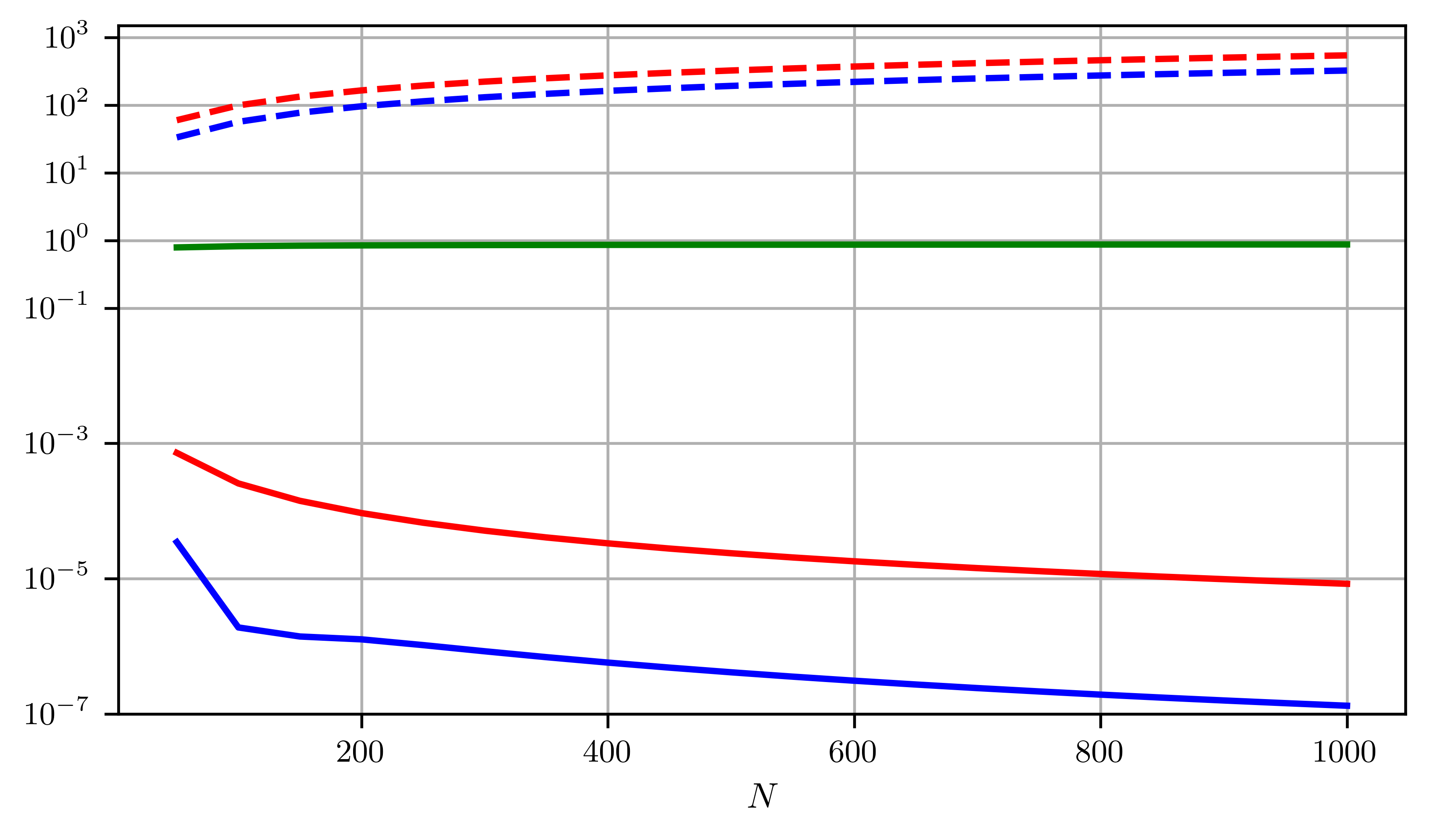}
    \caption{Comparison at the time $T=1$ starting from vacuum of the estimator $\xi_S(T)$ (dotted lines) with the differences between successive truncations of same parity (full lines) and between successive truncations of different parity (green line). More precisely, in red dotted lines is $\xi_S(T)$ for $\psi_{6N}$, while in blue dotted lines is $\xi_S(T)$ for $\psi_{6N+3}$. The full red lines are $\|\ket{\psi}_{6N}(T)-\ket{\psi}_{6(N + 1)}(T)\|$ and the full blue lines are $\|\ket{\psi}_{6N+3}(T)-\ket{\psi}_{6(N + 1)+3}(T)\|$. The green line is $\|\ket{\psi}_{6N}(T)-\ket{\psi}_{6N+3}(T)\|$.}
    \label{fig:fullwidth_example}
\end{figure}







\section{Numerical comparison with \cite{PRL_estimator}}
\label{sec:PRL_comp}
In this section we compare numerically our spatial-truncation estimator with the estimator derived in Woods et al.~\cite{PRL_estimator}. The test problem in~\cite{PRL_estimator} is a two-level system coupled to a bosonic bath modelled as a finite chain of $L$ linearly coupled harmonic oscillators. The total Hamiltonian reads
\begin{align*}
    \mathbf{H} &= -\Delta \tfrac{1}{2}\boldsymbol{\sigma}_x + \omega_c \sqrt{\tfrac{2\alpha}{s+1}}\,\tfrac{1}{2}\boldsymbol{\sigma}_z\,\mathbf{x}^0 + \mathbf{H}^B,
\end{align*}
with the bath Hamiltonian
\begin{align*}
    \mathbf{H}^B = \tfrac{1}{2}\sum_{i,j=0}^{L-1} \bigl(\mathbf{x}^i X_{ij} \mathbf{x}^j + \mathbf{p}^i P_{ij} \mathbf{p}^j\bigr),
\end{align*}
where $\mathbf{x}^i$ and $\mathbf{p}^i$ are the position and momentum operators of the $i$-th oscillator. We use the particle mapping (so that $X_{ij}=P_{ij}$), where the tridiagonal coefficients are
\begin{align*}
X_{i+1,i+1} &= \tfrac{\omega_c}{2}\Bigl(1 + \tfrac{s^2}{(s+2i)(s+2i+2)}\Bigr), \quad i=0,\dots,L-2, \\
X_{i,i+1} &= X_{i+1,i}= \omega_c\,\tfrac{(i+1)(i+1+s)}{(s+2i+2)(s+2i+3)}\sqrt{\tfrac{s+2i+3}{s+2i+1}},\quad i=0,\dots,L-2.
\end{align*}
As in Fig. 2 of~\cite{PRL_estimator}, we fix the model parameters to $\alpha=0.8$, $s=3$ and $\omega_c=\Delta=1$. The initial state is
\begin{equation*}
    \ket{\psi(0)} = \ket{\uparrow} \otimes {\ket{0}}^{\otimes L},
\end{equation*}
where $\ket{\uparrow}$ is the $+1$ eigenstate of $\boldsymbol{\sigma}_z$.
Their estimator bounds the error made on the expectation value of any bounded observable $\mathbf{O}$ on the system when truncating the Hilbert space of each oscillator to its first $m$ levels. Our estimator bounds the trace norm error between the exact solution on the full Hilbert space and the solution on the truncated space, i.e., system and truncated bath. Hence, our estimator, which by duality bounds the error on the expectation value of any bounded observable $\mathbf{O}$ on the total system, also bounds the error on the expectation value of any bounded observable on the system alone. That is, the error their estimator addresses is always encompassed by the (potentially larger) error our estimator bounds. Nevertheless, we observe that our estimator has better accuracy on this example than that of~\cite{PRL_estimator} as shown in~\cref{fig:woods_comparison}.
\begin{figure}[!htbp]
    \centering
    \begin{minipage}[b]{0.49\textwidth}
        \centering
        \includegraphics[width=\linewidth]{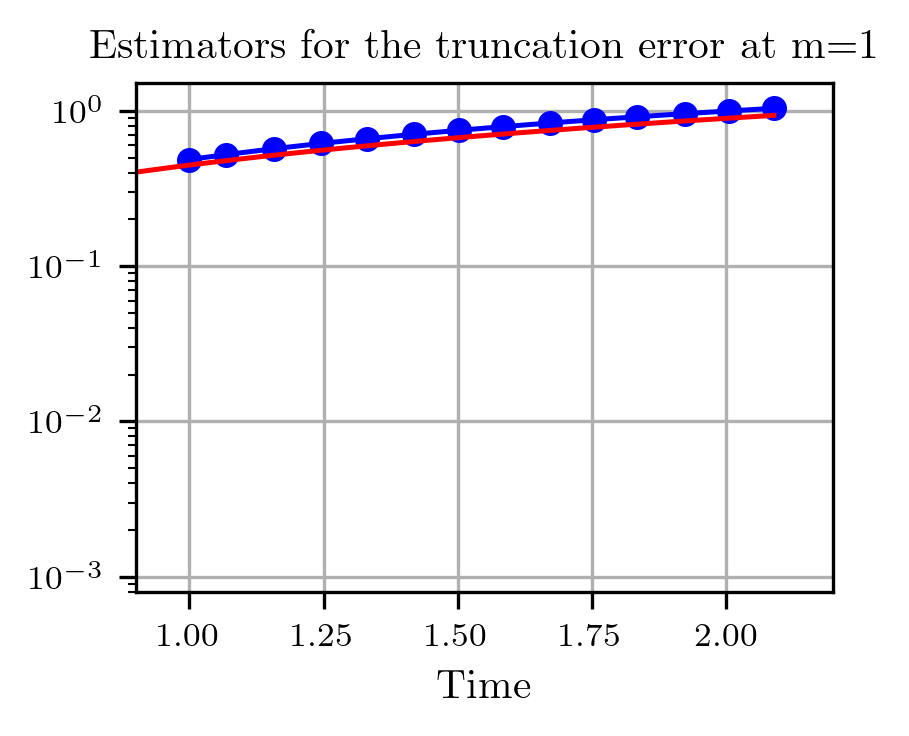}
    \end{minipage}  
    \begin{minipage}[b]{0.49\textwidth}
        \centering
        \includegraphics[width=\linewidth]{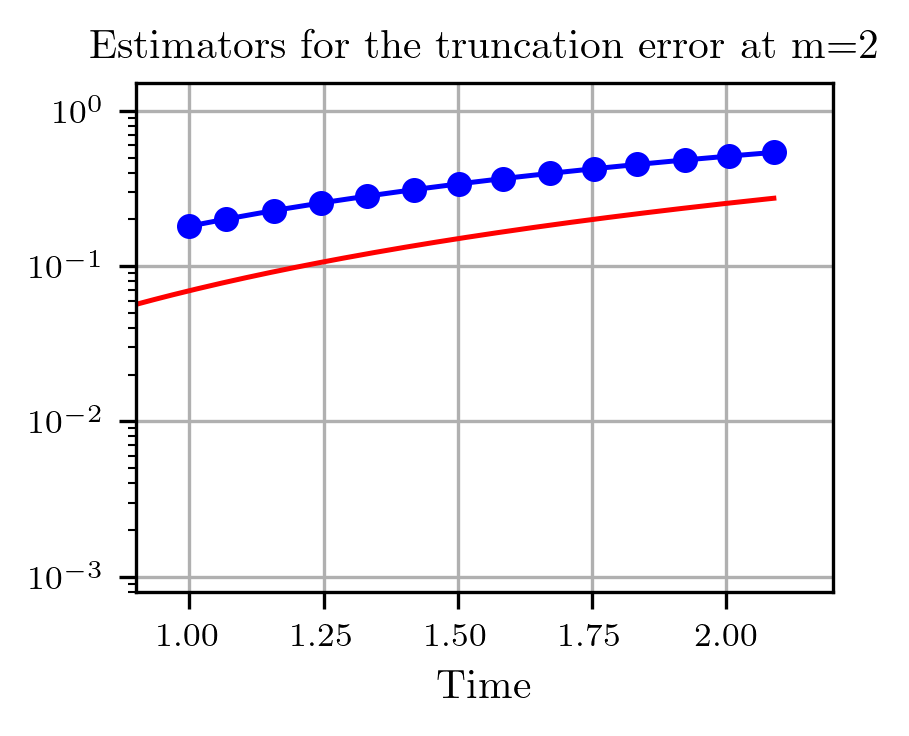}
    \end{minipage}

    \begin{minipage}[b]{0.49\textwidth}
        \centering
        \includegraphics[width=\linewidth]{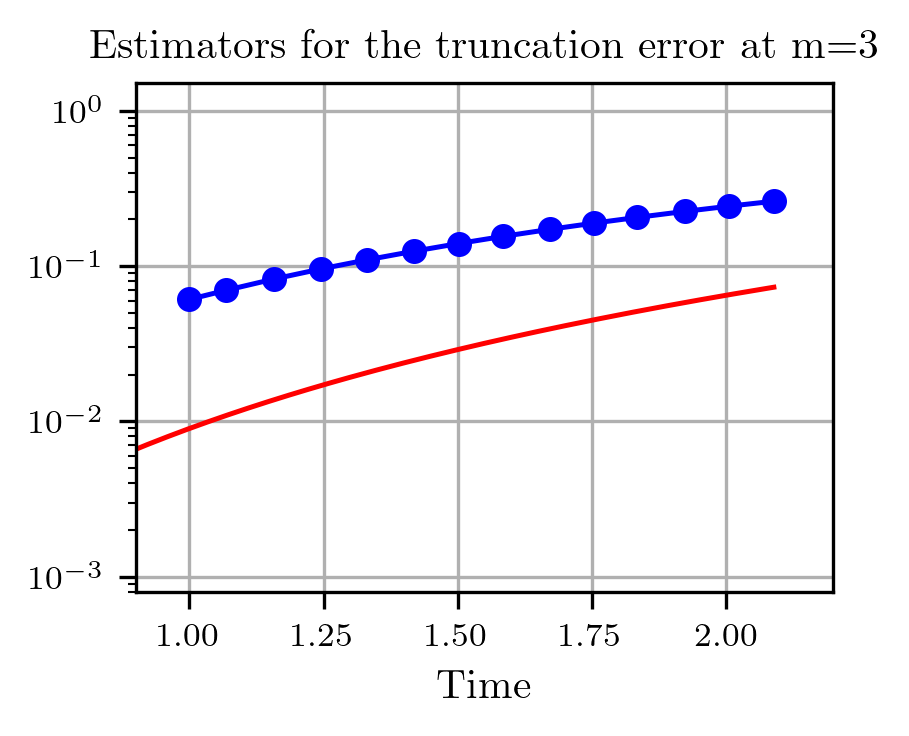}
    \end{minipage}
    \begin{minipage}[b]{0.49\textwidth}
        \centering 
        \includegraphics[width=\linewidth]{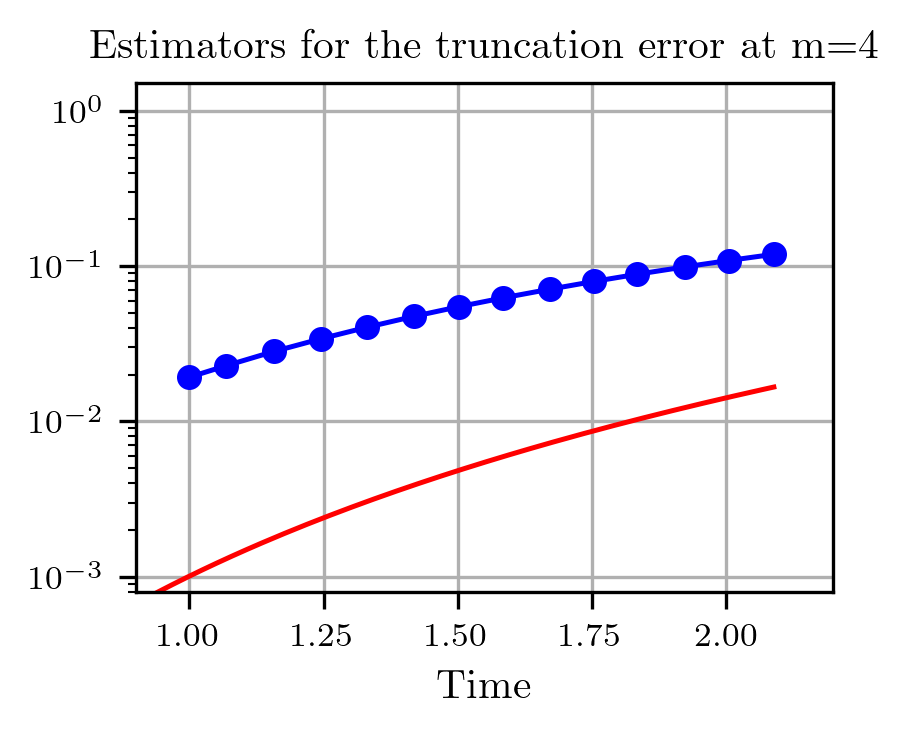}
    \end{minipage}
    \caption{Plots of the estimators obtained from~\cite{PRL_estimator}[Figure 2] in blue compared to our work, in red, at various truncations $m$ and for $L=3$ linearly coupled harmonic oscillators as described in the main text of \cref{sec:PRL_comp}.}
    \label{fig:woods_comparison}
\end{figure}

\section{Technical computations of \cref{ex:cat}}

\label{eq:cat}
First we want to simplify the expression of $(\cD_\LL-\cD_{\LL_N})(\boldsymbol{\rho}_N)$ for $\LL = \destroy^2- \alpha^2 \Id$. We decompose the three terms of the dissipator
\begin{align}
    \LL \boldsymbol{\rho}_N \LL^\dag-\LL_N \boldsymbol{\rho}_N \LL_N^\dag&=
    (\LL-\LL_N)\boldsymbol{\rho}_N(\LL^\dag-\LL_N^\dag)+ \LL_N\boldsymbol{\rho}_N(\LL^\dag-\LL_N^\dag)+(\LL-\LL_N)\boldsymbol{\rho}_N\LL_N^\dag,\\
    \LL^\dag \LL \boldsymbol{\rho}_N- \LL_N^\dag \LL_N \boldsymbol{\rho}_N&=
    (\LL^\dag-\LL_N^\dag) (\LL-\LL_N)\boldsymbol{\rho}_N+ (\LL^\dag-\LL_N^\dag) \LL_N\boldsymbol{\rho}_N+\LL_N^\dag (\LL-\LL_N)\boldsymbol{\rho}_N,\\
    \boldsymbol{\rho}_N \LL^\dag \LL-\boldsymbol{\rho}_N  \LL_N^\dag \LL_N &=(\LL^\dag \LL \boldsymbol{\rho}_N- \LL_N^\dag \LL_N \boldsymbol{\rho}_N)^\dagger .
\end{align}
For $\LL = \destroy^2- \alpha^2 \Id$ we also have:
\begin{align}
    \notag (\LL-\LL_N)\Proj{N}&=((\destroy^2-\alpha^2\Id)-(\Proj{N}\alpha^2\Id\Proj{N}-\Proj{N}\alpha^2\Id\Proj{N}))\Proj{N},\\
    &=0, \qquad \text{using } \Proj{N} \destroy=\destroy_N.
\end{align}
Thus, we obtain
\begin{align}
    \LL \boldsymbol{\rho}_N \LL^\dag-\LL_N \boldsymbol{\rho}_N \LL_N^\dag&=0,\\
    \LL^\dag \LL \boldsymbol{\rho}_N- \LL_N^\dag \LL_N \boldsymbol{\rho}_N&=(\LL^\dag-\LL_N^\dag) \LL_N\boldsymbol{\rho}_N=(\boldsymbol{\rho}_N \LL_N^\dag (\LL-\LL_N))^\dagger,\\
    \boldsymbol{\rho}_N \LL^\dag \LL-\boldsymbol{\rho}_N  \LL_N^\dag \LL_N &=\boldsymbol{\rho}_N \LL_N^\dag (\LL-\LL_N).
\end{align}
Using $\LL = \destroy^2- \alpha^2 \Id$, one has
\begin{align}
    \notag \LL_N^\dag (\LL-\LL_N)&=(\Proj{N} \create2\Proj{N} - \Proj{N}\alpha^2\Id\Proj{N})((\destroy^2 - \alpha^2\Id) - (\Proj{N} \destroy^2\Proj{N} - \Proj{N}\alpha^2\Id\Proj{N}))\\
    \notag &=-\Proj{N} \create2 \Proj{N} \destroy^2 \Proj{N} + \Proj{N} \create2 \Proj{N} \alpha^2\Id \Proj{N} + \Proj{N} \create2 \Proj{N} \destroy^2 - \Proj{N} \create2 \Proj{N} \alpha^2\Id\\
    \notag & +  \Proj{N}\alpha^2\Id\Proj{N} \destroy^2\Proj{N} - \Proj{N}\alpha^2\Id\Proj{N}\alpha^2\Id\Proj{N} - \Proj{N}\alpha^2\Id\Proj{N} \destroy^2 + \Proj{N}\alpha^2\Id\Proj{N}\alpha^2\Id\\
    &=-\alpha^2 \sqrt{N(N+1)}\ket{N-1}\bra{N+1}-\alpha^2 \sqrt{(N+1)(N+2)}\ket{N}\bra{N+2}.
\end{align}
Hence, 
\begin{align}
    \notag (\cD_\LL-\cD_{\LL_N})(\boldsymbol{\rho}_N)&=  -\frac{\alpha^2}{2}  \Big( \sqrt{(N\!+\!1)(N\!+\!2)} \notag\\
    &\quad\times \big( \ket{N\!+\!2}\bra{N} \rN(t) + \rN(t) \ket{N}\bra{N\!+\!2} \big)  \\
    & + \sqrt{N(N\!+\!1)} \big( \ket{N\!+\!1}\bra{N\!-\!1} \rN(t) \notag\\
    &\quad + \rN(t) \ket{N\!-\!1}\bra{N\!+\!1} \big) \Big).
\end{align}
Next, $(\cD_\LL-\cD_{\LL_N})(\boldsymbol{\rho}_N)$ is an operator of rank lower than $4$, as it is supported on
\[
\{ \rN \ket{N},\; \rN \ket{N-1},\; \ket{N+1},\;\ket{N+2}\}
\]
and one can efficiently compute numerically its trace norm.
The square of this operator is:
\begin{align}
    \notag ((\cD_\LL-\cD_{\LL_N})(\boldsymbol{\rho}_N))^2=&  \frac{\alpha^4}{4}  \big((N\!+\!1)(N\!+\!2) \big( \ket{N\!+\!2}\bra{N} \rN^2(t)\ket{N}\bra{N\!+\!2} \notag\\
    &\quad + \rN(t)\ket{N}\bra{N}\rN(t)\big)\\
    &+ N(N\!+\!1) \big( \ket{N\!+\!1}\bra{N\!-\!1} \boldsymbol{\rho}^2_N(t)\ket{N\!-\!1}\bra{N\!+\!1}\notag\\
    &\quad + \rN(t) \ket{N\!-\!1}\bra{N\!-\!1} \rN(t)\big)\big).
\end{align}
It remains to compute the following part, expressed in the decomposition $\xH= \xH_N \oplus \xH_N^\perp$:
\begin{align}
    \notag &\|(\cD_\LL-\cD_{\LL_N})(\boldsymbol{\rho}_N)\|_1 = \sqrt{N+1}\frac{\alpha^2}{2} \\
    &\times \xtr{\sqrt{ \left(
        \begin{array}{cc}
            \begin{aligned}
                &N\rN \ket{N\!-\!1}\bra{N\!-\!1} \rN\\
                & + (N\!+\!2) \rN\ket{N}\bra{N}\rN
            \end{aligned}
            & 0 \\
        0
        &\begin{aligned}
            &(N\!+\!2) \ket{N\!+\!2}\bra{N} \rN^2\ket{N}\bra{N\!+\!2}\\
            & + N \ket{N\!+\!1}\bra{N\!-\!1} \boldsymbol{\rho}^2_N\ket{N\!-\!1}\bra{N\!+\!1}
        \end{aligned}
        \end{array} \right)}
    }.
\end{align}
It is again block diagonal so we can simplify it as follows
\begin{align}
    \notag &\|(\cD_\LL-\cD_{\LL_N})(\boldsymbol{\rho}_N)\|_1 = \sqrt{N\!+\!1}\frac{\alpha^2}{2}\Biggl(\xtr{\sqrt{
        \begin{aligned}
            &N\rN \ket{N\!-\!1}\bra{N\!-\!1} \rN\\
            & + (N\!+\!2) \rN\ket{N}\bra{N}\rN
        \end{aligned}}
        }\\
    \notag & + \xtr{\sqrt{
        \begin{aligned}
            &(N\!+\!2) \ket{N\!+\!2}\bra{N} \rN^2\ket{N}\bra{N\!+\!2}\\
            & + N \ket{N\!+\!1}\bra{N\!-\!1} \boldsymbol{\rho}^2_N\ket{N\!-\!1}\bra{N\!+\!1}
        \end{aligned}
    }}\Biggr)\\
    \notag &=\sqrt{N+1}\frac{\alpha^2}{2} \Biggl(\xtr{\sqrt{\rN \Big(\sqrt{N}\ket{N-1}\bra{N-1} + \sqrt{N+2}\ket{N}\bra{N}\Big)\rN}}\\
    & + \Big(\sqrt{N}\bra{N-1}\rN^2\ket{N-1} + \sqrt{N+2}\bra{N}\rN^2\ket{N}\Big) \Biggr).
\end{align}

\printbibliography
\end{document}